\documentclass[a4paper]{article}
\pdfoutput=1
\usepackage{graphicx}
\usepackage{amsmath}
\usepackage{amsfonts}
\usepackage{amsthm}
\usepackage{amssymb}
\usepackage{color}
\usepackage{geometry}
\usepackage{bbm}
\usepackage{tikz}
\usetikzlibrary{arrows}
\usepackage{todonotes}
\usepackage{verbatim}

\newcommand{\dx}{\, \mathrm{d}x}

\newcommand{\N}{\mathcal{N}}
\newcommand{\T}{\mathcal{T}}

\newcommand{\V}{\mathcal{V}}

\newcommand{\supp}{\text{supp}}

\newcommand{\CP}{C_{\mathrm{PF}}}

\newcommand{\uref}{u_h}
\newcommand{\uH}{u_H}
\newcommand{\uHL}{u^L_H}
\newcommand{\VGammah}{\V^\Gamma_h}

\newtheorem{theorem}{Theorem}[section]

\newtheorem{lemma}[theorem]{Lemma}

\theoremstyle{definition}
\newtheorem{definition}[theorem]{Definition}
\newtheorem{remark}[theorem]{Remark}
\newtheorem{assumption}[theorem]{Assumption} %

\numberwithin{equation}{section}

\begin{document}

\begin{center}
  {\LARGE Multiscale methods for problems with complex geometry}\\[2em]
\end{center}

\renewcommand{\thefootnote}{\fnsymbol{footnote}}

\begin{center}
{\large Daniel Elfverson\footnote[1]{Department of Mathematics, Ume{\aa} University, SE-901 87 Ume{\aa}, Sweden. Supported by SSF.}, Mats G. Larson\footnote[2]{Department of Mathematics, Ume{\aa} University, SE-901 87 Ume{\aa}, Sweden. Supported by SSF.}, and Axel M{\aa}lqvist\footnote[3]{Department of Mathematical Sciences, Chalmers University of Technology and University of
Gothenburg SE-412 96 G{\"o}oteborg, Sweden. Supported by the Swedish research council and SSF.}
 \renewcommand{\thefootnote}{\arabic{footnote}}\setcounter{footnote}{2}}\\[2em]
\end{center}

\begin{center}
{\large{\today}}
\end{center}

\begin{center}
\end{center}

\begin{abstract}
  We propose a multiscale method for elliptic problems on
  complex domains, e.g. domains with cracks or complicated boundary. For local
  singularities this paper also offers a discrete alternative to enrichment
  techniques such as XFEM. We construct corrected coarse test and trail spaces
  which takes the fine scale features of the computational domain into
  account. The corrections only need to be computed in regions surrounding
  fine scale geometric features. We achieve linear convergence rate in energy
  norm for the multiscale solution. Moreover, the conditioning of the
  resulting matrices is not affected by the way the domain boundary cuts the
  coarse elements in the background mesh. The analytical findings are verified
  in a series of numerical experiments.
\end{abstract}

\section{Introduction}

Partial differential equations with data varying on multiple scales in space
and time, so called {\em multiscale problems}, appear in many areas of science
and engineering. Two of the most prominent examples are composite materials
and flow in a porous medium. Standard numerical techniques may perform
arbitrarily bad for multiscale problems, since the convergence rely on
smoothness of the solution \cite{BaOs00}. Also adaptive techniques
\cite{Ve96}, where local singularities are resolved by local mesh refinement,
fail for multiscale problems since the roughness of the data is often not
localized in space. As a remedy against this issue generalized finite element
methods and other related multiscale techniques have been
developed
\cite{BaCaOs94,HeFeMaQu98,LaMa07,HoWu97,WeEn03,LaMa07,Ma11,MaPe14,OwZhBe14}.
So far these techniques have focused on multiscale coefficients in general and
multiscale diffusion in particular. Significantly less work has been directed
towards handling a computational domain with multiscale boundary. However, in
many applications including voids and cracks in materials and rough surfaces,
multiscale behavior emanates from the complex geometry of the computational
domain. Furthermore, the classical multiscale methods mentioned above aim at,
in different ways, upscaling the multiscale data to a coarse scale where it is
possible to solve the equation to a reasonable computational cost. However,
these techniques typically assume that the representation of the computational
domain is the same on the coarse and fine scale. In practice this is very
difficult to achieve unless the computational domain has a simple shape, which
is not the case in many practical applications.

In this paper we design a multiscale method for
problems with complex computational domain. In order to simplify the
presentation we neglect multiscale coefficients in the analysis even
though the methodology directly extends to this situation. The proposed
algorithm is based on the localized orthogonal decomposition (LOD) technique
presented in \cite{MaPe14} and further developed in
\cite{ElGeMa13,ElGeMaPe13,MaPe15,Peterseim:2014}. In LOD both test and trail
spaces are decomposed into a multiscale space and a reminder space that are
orthogonal with respect to the scalar product induced by the bilinear form
of the problem considered. In this paper we propose and
analyze how the modified multiscale basis functions can be blended with
standard finite element basis functions, allowing them to be used only close
to the complex boundary. We prove optimal convergence and show that the
condition number of the resulting coarse system of equations scales at an
optimal rate with the mesh sizes. The gain of this approach is that the global
solution is computed on the coarse scale, with the accuracy of the fine scale.
Also, all localized fine scale computations needed to enrich the standard
finite element basis are localized and can thus be done in parallel.

Other work on multiscale methods for complex/rough domains are
\cite{Mad09,MV06}, which are based on the multiscale finite element method
(MsFEM). However, the analysis in \cite{Mad09,MV06} is limited to periodic data. An explicit methods to handle
complex geometry is the cut finite element method (CutFEM) \cite{BCHLM14}
which use a robust Nitsche's formulation to weakly enforce the
boundary/interface conditions. See also \cite{GPP94} for an other fictitious domain methods. An
other explicit approach is the composite finite element method \cite{HS97}
which constructs a coarse basis that is fitted to the boundary. The
fictitious domain and composite finite element approaches are useful in e.g.
multigrid methods since they construct a coarse representation on domains with
fine scale features, see also \cite{Yse99}. The technique proposed in this paper is more
related to the extended finite element method (XFEM) \cite{FB10} where the
polynomial approximation space is enriched with {non-polynomial} functions.
The method can be used as a discrete alternative to XFEM, that can be
useful e.g. when the nature of the singularities are unknown.

The outline of the paper is as follows.  In Section~\ref{preliminaries} we
present the model problem and introduce some notation.  In
Section~\ref{multiscale} we formulate a multiscale method for problems where
the mesh does not resolve the boundary. In Section~\ref{error} we analyse the
proposed method in several different steps and finally prove a bound of the
error in energy norm, which shows that the error is of the same order as the
standard finite element method on the coarse mesh for smooth problem. In
Section~\ref{implement} we shortly describe the implementation of the method
and prove a bound of the condition number of the stiffness matrix. In Section~\ref{numerics} we present some numerical
experiment to verify the convergence rate and conditioning of the proposed
method. Finally in the Appendix we prove a technical Lemma.

\section{Preliminaries}
\label{preliminaries}
In this section we present a model problem, introduce some notation,
and define a reference finite element discretization of the model problem.
\subsection{Model problem}
We consider the Poisson equation in a bounded polygonal/polyhedral domain
$\Omega\subset\mathbb{R}^d$ for $d=2,3,$ with a complex/fine scale
boundary $\partial\Omega=\Gamma_D\cup\Gamma_R$. That is, we consider
\begin{equation}
  \begin{aligned}
    -\Delta u &= f\quad&\text{in }&\Omega, \\
    \nu\cdot \nabla u + \kappa u& = 0 \quad&\text{on }&\Gamma_R, \\
    u & = 0 \quad&\text{on }&\Gamma_D,
  \end{aligned}
\end{equation}
where $\nu$ is the exterior unit normal of $\partial\Omega$,
$0\leq \kappa\in\mathbb{R}$, and $f\in L^2(\Omega)$. For simplicity we
assume that, if $\kappa=0$ then $\Gamma_D\neq\emptyset$ to guarantee existence
and uniqueness of the solution $u$.  The weak form of the partial
differential equation reads: find
$u\in \V:= \{v\in H^1(\Omega)\mid v|_{\Gamma_D}=0\}$ such that
\begin{equation}\label{eq:week}
  a(u,v):= \int_{\Omega}\nabla u\cdot\nabla v \dx + \int_{\Gamma_R}\kappa u v \,\mathrm{d}S = \int_\Omega f v\dx  =: F(v),
\end{equation}
for all $v\in \V$. Throughout the paper we use standard notation for
Sobolev spaces \cite{Adams}. We denote the local energy and $L^2$-norm
in a subset $\omega\subset\Omega$ by
\begin{equation}
  ||| v |||_\omega =  \left(\int_{\omega}|\nabla u|^2 \dx + \int_{\Gamma_R\cap\partial\omega}\kappa u^2 \,\mathrm{d}S \right)^{1/2},
\end{equation}
and
\begin{equation}
  \| v \|_\omega =  \left(\int_{\omega}u^2 \dx\right)^{1/2},
\end{equation}
respectively. Moreover, if $\omega=\Omega$ we omit the subscript,
$||| v |||:=||| v |||_\Omega$ and $\| v \|:=\| v \|_\Omega$.

\subsection{The reference finite element method}
\label{sec:reference}

We embed the domain $\Omega$ in a polygonal domain $\Omega_{0}$
equipped with a quasi-uniform and shape regular mesh $\T_{H,0}$, i.e.,
$\Omega\subset\Omega_{0}$ and
$\overline\Omega_{0}=\sum_{T\in\T_{H,0}}\overline T$.  We let $\T_H$
be the sub mesh of $\T_{H,0}$ consisting of elements that are cut or
covered by the physical domain $\Omega$, i.e,
\begin{equation}
  \T_H = \{T\in\T_{H,0} \mid T\cap \Omega\neq \emptyset\}.
\end{equation}
The finite element space on $\T_H$ is defined by
\begin{equation}
  \V_H=\{v\in\mathcal{C}^0(\Omega)\mid \forall T\in\T_H\text{, }v|_T\in\mathcal{P}_1(T)\},
\end{equation}
where $\mathcal{P}_1(T)$ is the space of polynomials of total degree
$\leq 1$ on $T$. We have $\V_H=\text{span}\{\varphi_x\}_{x\in\mathcal{N}}$, where $\mathcal{N}$ is the set of all nodes in the mesh
$\T_H$ and $\varphi_x$ is the linear nodal basis function
associated with node $x\in \mathcal{N}$.

The space $\V_H$ will not be sufficiently fine to represent the boundary data.
We therefore enrich the space $\V_H$ close to the boundary $\partial\Omega$.
In order to construct the enrichment we define $L$-layer patches around the
boundary recursively as follows
  \begin{equation}
    \begin{aligned}
      \omega^0_\Gamma &:= \text{int}\left(( \bar T\in\T_H\mid T\cap \Omega\neq T)\cap\Omega\right), \\
      \omega_\Gamma^\ell &:= \text{int}\left((\bar T \in\T_H\mid \bar T\cap \overline\omega_\Gamma^{\ell-1}\neq \emptyset)\cap\Omega\right),\quad\text{for }\ell = 1,\dots,L. \\
    \end{aligned}
  \end{equation}
Note that $\omega_\Gamma^0$ is the set of all elements which are cut by the
domain boundary $\partial\Omega$. An illustration of
$\Omega$, $\T_{H,0}$, $\omega_\Gamma^0$, and $\omega_\Gamma^1$ are given in
  Figure~\ref{fig:referencespace}.
 \begin{figure}
   \centering
   \includegraphics[width=0.6\textwidth]{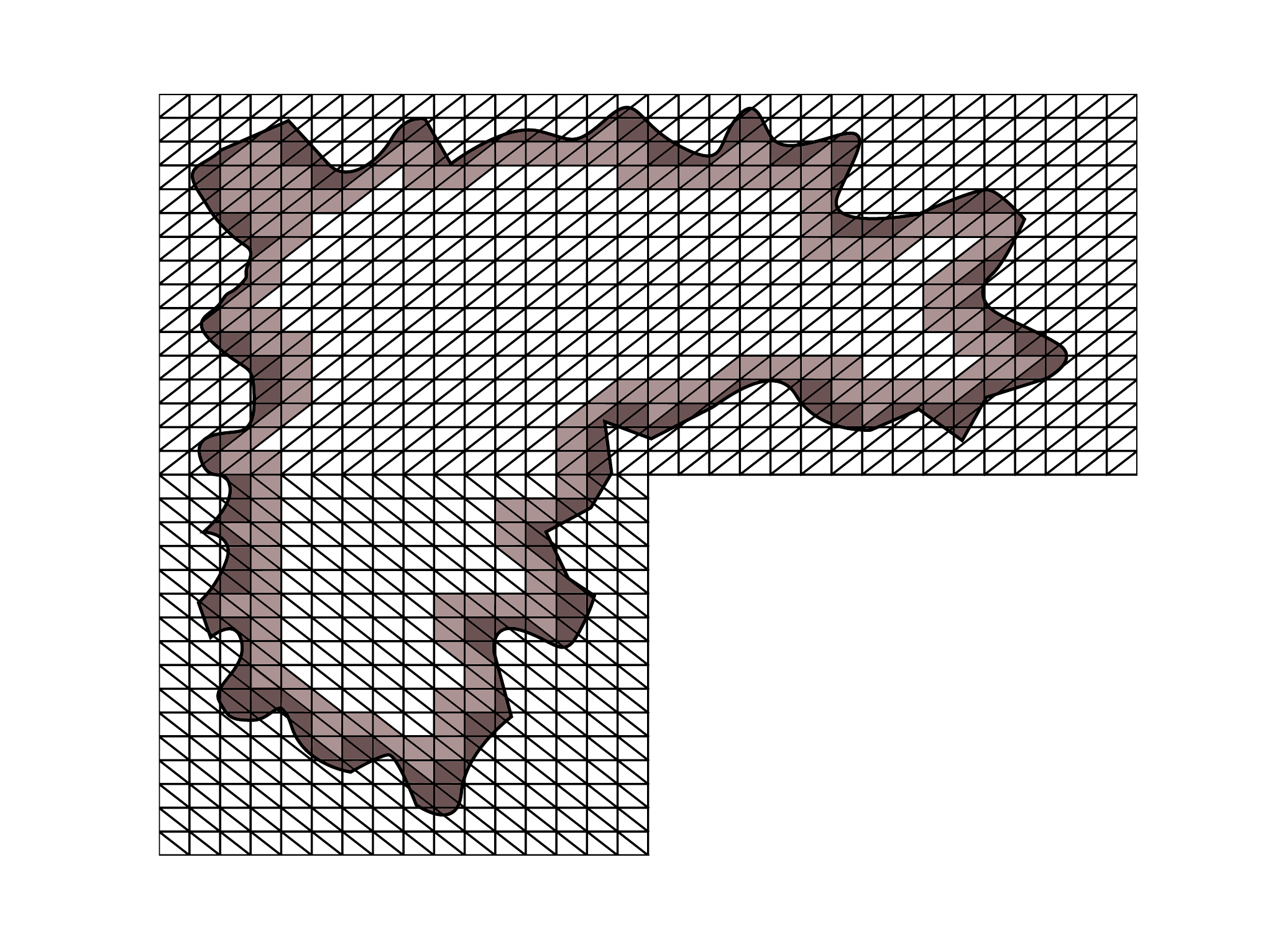}
   \caption{The computational domain $\Omega$ embedded in the mesh
     $\T_{H,0}$. The dark grey area is $\omega_\Gamma^0$, and the union
     of the dark and light grey area is $\omega_\Gamma^1$.}
   \label{fig:referencespace}
 \end{figure}
 We will later see, in Lemma~\ref{lemma:H1enrichment}, that the
 appropriate number of layers is determined by the decay of the
 $H^2$-norm of the exact solution away from the boundary.

 Let $\T_h$ be a fine mesh defined on $\omega_\Gamma^k$, obtained by refining
 the coarse mesh in $\omega_\Gamma^k$. On the interior part of the boundary
 $\partial\omega_\Gamma^k \setminus \partial \Omega$ we allow hanging nodes.
 We define $\V_h(\omega_\Gamma^k)=\{ v\in\mathcal{C}^0(\omega_\Gamma^k) \mid
 \forall T\in\T_h\text{, }v|_T\in\mathcal{P}_1(T) \}$ and a reference finite
 element space by
\begin{equation}
  \VGammah := (\V_H + \V_h(\omega_\Gamma^k))\cap H^1_{\Gamma_D}(\Omega),
\end{equation}
which consists of the standard finite element space enriched with a
locally finer finite element space in $\omega_\Gamma^k$. We assume that
the space $\VGammah$ is fine enough to resolve the boundary, i.e., we assume that the boundary $\partial \Omega$ is exactly represented by the fine mesh $\T_h$.

The finite element method posed in the enriched space $\VGammah$
reads: find $\uref\in\VGammah$ such that
\begin{equation}\label{eq:omegagamma}
  a(\uref,v) = F(v)\quad\text{for all }v\in\VGammah.
\end{equation}
We call the solution to (\ref{eq:omegagamma}) the reference solution.
We have the following a priori error estimate
\begin{equation}\label{eq:uref}
  ||| u - \uref ||| \leq C(H|u|_{H^2(\Omega\setminus\omega^{k-1}_\Gamma)} + h^{s-1}|u|_{H^s(\omega^{k-1}_\Gamma)})
\end{equation}
where $1\leq s\leq 2$ depends on the regularity of $u$ in $\omega^{k-1}_\Gamma$. For a
proof of \eqref{eq:uref} see Section~\ref{sec:estaimte_uref}.

\section{The multiscale method}
\label{multiscale}
In the multiscale method we want to construct a coarse scale
approximation of $u_h$, which can be computed at a low cost. We present the
method in two steps:
\begin{itemize}
\item First, we construct a global multiscale method using a
  corrected coarse basis which takes the fine scale variation of the
  boundary into account.
\item Then, we construct a localized multiscale method where the
  corrected basis is computed on localized patches.
\end{itemize}

\subsection{Global multiscale method}
For each $x \in\mathcal{N}$ (the set of free nodes) we define a
$L$-layer nodal patch recursively by letting
  \begin{equation}
    \begin{aligned}
      \omega_x^0 &=: \text{int}\left((\bar T \in\T_H\mid \bar T\cap x\neq \emptyset)\cap\Omega\right), \\
      \omega_x^\ell &=: \text{int}\left((\bar T \in\T_H\mid\bar T\cap \overline \omega_x^{\ell-1}\neq \emptyset)\cap\Omega\right),\quad\text{for }\ell= 1,\dots,L.
    \end{aligned}
  \end{equation}
We consider a projective Cl\'{e}ment interpolation operator defined by
\begin{equation}\label{eq:clement}
  \mathcal{I}_Hv = \sum_{x\in\N_{I}}(P_xv)(x)\varphi_x,
\end{equation}
where $\N_I$ is the index set of all interior nodes in $\Omega$ and
 $P_x$ is a local $L^2$-projection defined by: find
$P_xv\in\{v\in\V_H\mid \text{supp}(v)\cap\omega_x^0\neq \emptyset \}$ such
that
\begin{equation}
  (P_xv,w)_{\omega_x^0} = (v,w)_{\omega_x^0}\quad\text{for all }w\in \{v\in\V_H\mid \text{supp}(v)\cap\omega_x^0\neq\emptyset \}.
\end{equation}
This is not the same operator as proposed in the original LOD paper
\cite{MaPe14}. We choose the projective Cl\'{e}ment interpolation operator
since projective property simplifies the analysis and that it is slightly more
stable for multiscale problems \cite{BrPe14}.

Using the interpolation operator we split the space $\VGammah$
into the range and the kernel of the interpolation operator, i.e.,
$\V_H=\mathcal{I}_H\VGammah$ and $\V^f=(1-\mathcal{I}_H)\VGammah$.
Since the space $\V_H$ does not have the sufficient approximation
properties we use the same idea as in LOD and construct an
orthogonal splitting with respect to the bilinear form. We define
the corrected coarse space as
\begin{equation}
  \V_H^\Gamma = (1+Q)\mathcal{I}_H\VGammah
\end{equation}
where the operator $Q$ is defined as follows: given $v_H\in \V_H$ find
$Q(v_H)\in\{v\in \V^f \mid v|_{\Gamma_D}=-v_H\}$ such that
\begin{equation}\label{eq:global_corrector}
  \begin{aligned}
  a(Q(v_H),w) &= -a(v_H,w)\quad\text{for all }v\in\{v\in \V^f \mid w|_{\Gamma_D}=0\}.
    \end{aligned}
\end{equation} Note that $\V_H\not\subset \VGammah$ but $\V^\Gamma_H\subset
\VGammah$ because the correctors are solved with boundary conditions that
compensates for the nonconformity of the space $\V_H$ and also that
$Q(v_H)|_{\Omega\setminus\omega_\Gamma^k}=0$ since $\V^f$ only has support in
$\omega_\Gamma^k$. From
\eqref{eq:global_corrector} we have the orthogonality $a(\V^\Gamma_H,\V^f)=0$
and we can write the reference space as the direct sum
$\VGammah=\V^\Gamma_H\oplus_a\V^f$, where the orthogonality is with respect to the bilinear form $a$.

The multiscale method posed in the space $\V_H^\Gamma$ reads: find
$\uH \in \V_H^\Gamma$ such that
\begin{equation}\label{eq:idealgamma}
  a(\uH,v) = F(v)\quad\text{for all }v\in \V_H^\Gamma.
\end{equation}

\begin{remark}\label{rem:global_multiscale_method}
Note that, even in the global multiscale method the support of the correctors
are only in $\omega_\Gamma^k$ which we defined in Section~\ref{sec:reference}.
\end{remark}

\subsection{Localized multiscale method}
Finally, we further localize the computation of the corrected basis functions
to nodal patches on $\omega_\Gamma^k$. Using linearity of the operator
$Q(v_H)$, we obtain
\begin{equation}
  Q(v_H) = \sum_{x\in\N_I}v_H(x)Q(\varphi_x).
\end{equation}
We denote the localized corrector by
\begin{equation}
  Q^L(v_H) = \sum_{x\in\N_I}v_H(x)Q^L_x(\varphi_x).
\end{equation}
where $Q^L_x(\varphi_x)$ is the localization of $Q(\varphi_x)$
computed on an $L$-layer patch. The local correctors are computed
as follows: given $x\in\mathcal{N}_I$ find
$Q^L_x(\varphi_x)\in \{v\in\V^\mathrm{f}\mid
v|_{\Omega\setminus\omega^L_x}=0\text{ and }v|_{\Gamma_D}=
-\varphi_x\}$ such that,
\begin{equation}\label{eq:local_corrector}
  \begin{aligned}
    a( Q^L_x(\varphi_x),w) &= -a( \varphi_x,w)\quad\text{for all }v\in \{v\in\V^\mathrm{f}\mid
v|_{\Omega\setminus\omega^L_x}=0\text{ and }v|_{\Gamma_D}=0\}.
  \end{aligned}
\end{equation}
The localized multiscale method reads: find
$\uHL\in\V^{\Gamma,L}_H:=\text{span}\{\varphi_x +
Q^L_x(\varphi_x)\}_{x\in\mathcal{N}_I}$ such that
\begin{equation}\label{eq:local_multiscale}
  a(\uHL,v) = F(v) \quad\text{for all }v\in\V^{\Gamma,L}_H.
\end{equation}
The space $\V^{\Gamma,L}_H$ has the same dimension as the coarse space $\V_H$
but the basis functions have slightly larger support. The multiscale solution
$\uHL\in\V^{\Gamma,L}_H$ has better approximation properties than the standard
finite element solution on the same mesh satisfy
\begin{equation}
  |||\uref - \uHL||| \leq C_1 H,
\end{equation}
where $H$ is the mesh size, $C_1$ is a constant independent of the fine scale
features of the boundary $\Gamma$, and $L=\lceil C_2 \log_2(H^{-1})\rceil$ for
a constant $C_2$ ($=1.5$ in the numercal experminents). A proof is given in
Theorem \ref{thm:local_enriched_fem} in the coming section.

\section{Error estimates}
\label{error}
In this section we derive our main error estimates. First we present
some technical tools needed to prove the main result:
\begin{itemize}
\item We present an explicit way to compute an upper bound for
  Poincar\'e-Friedrichs constants on complex domains.
\item We prove approximation properties of the interpolation operator
  on these domains.
\end{itemize}
The main result is obtained in four steps:
\begin{itemize}
\item We bound the difference between the analytic solution and the reference finite element solution $|||u-\uref|||$.
\item We bound the difference between the reference finite element
solution and the ideal multiscale method $|||\uref-\uH|||$.
\item We bound the difference between a function $v\in\V_H$ modified
  by the global corrector and localized corrector $|||Q(v)-Q^L(v)|||$.
\item Together these properties are used to estimate the error between
  the analytic solution and the localized multiscale approximation.
\end{itemize}
Furthermore, let  $a\lesssim b$ abbreviate the inequality $a\leq C b$
where $C$ is any generic positive constant independent on the domain
$\Omega$ and of the the coarse and fine mesh sizes $H,h$.

\subsection{Poincar{\'e}-Friedrichs inequality on complex domains}
\label{sec:PF}
A crucial part of the proof is  a Poincar\'e-Friedrichs inequality with a constant of moderate size. The inequality reads: for all $u\in H^1(\omega)$ it holds,
\begin{equation}
  \inf_{c\in\mathbb{R}}\|u - c\|_{\omega}\leq C (\omega)\text{diam}(\omega)\|\nabla u\|_{\omega},
\end{equation}
where the optimal constant is
\begin{equation}
  c = \frac{1}{|\omega|}\int_{\omega}u\dx.
\end{equation}
Following \cite{PeSc12}, we consider inequalities of the following
type: for all $u\in H^1(\omega)$
\begin{equation}\label{eq:PFX}
  \|u - \lambda_\gamma(u)\|_{\omega}\leq C (\omega)\text{diam}(\omega)\|\nabla u\|_{\omega},
\end{equation}
where $\gamma\subset\partial\omega$ is a $(d-1)$-dimensional manifold
and
\begin{equation}
  \lambda_\gamma(u) = \frac{1}{|\gamma|}\int_{\gamma}u\,\mathrm{d}S.
\end{equation}
We introduce the notation $\CP=C(\omega)$, and refer to $\CP$ as the
Poincar\'e-Friedrichs constant, which depends on the domain $\omega$
but not on its diameter.

A direct consequence of \eqref{eq:PFX} is the following inequality
\begin{equation}\label{eq:general_PF}
  \|u\|_{\omega}\lesssim \CP \text{diam}(\omega)\|\nabla u\|_{\omega} + \text{diam}(\omega)^{1/2}\|u\|_{\gamma},
\end{equation}
which holds if $ \mathrm{diam}(\omega)^{d-1}\lesssim |\gamma|$, i.e., the average is taken over a large enough manifold $\gamma\subset \partial \omega$. A short proof is given by
\begin{equation}
  \begin{aligned}
    \|u\|_{\omega} &\leq \|u - \lambda_\gamma(u)\|_{\omega} + \|\lambda_\gamma(u)\|_{\omega} \\
    &\leq \CP\text{diam}(\omega)\|\nabla u\|_{\omega} + |\lambda_\gamma(u)|\cdot|\omega|^{1/2} \\
    &\leq \CP\text{diam}(\omega)\|\nabla u\|_{\omega} + \|\lambda_\gamma(u)\|_{\gamma} |\gamma|^{-1/2}|\omega|^{1/2} \\
    &\lesssim \CP\text{diam}(\omega)\|\nabla u\|_{\omega} + \text{diam}(\omega)^{1/2}\|\lambda_\gamma(u)\|_{\gamma}.
  \end{aligned}
\end{equation}
Furthermore, from \cite{PeSc12} we have the bound $\CP\leq 1$ for the
Poincar{\'e} constant on a $d$-dimensional simplex where $\gamma$ is
one of the facets.

Next we will review some results given in \cite{PeSc12}
applied to domains with complex boundary. In \cite{PeSc12} the notion
of quasi-monotone paths is use to prove weighted Poincar\'e-Friedrichs
type inequalities using average on $(d-1)$-dimensional manifolds
$\gamma\subset \omega$. These results have also been discussed for
perforated domains in \cite{BrPe14}.
\begin{definition}\label{def:path}
  For simplicity we assume that $\omega$ is a polygonal domain that is
  subdivided into a quasi-uniform partition of simplices
  $\tau=\{T_\ell\}_{\ell=1}^n$.  We call the region
  $P_{\ell_1,\ell_2}=(\overline T_{\ell_1}\cup \overline T_{\ell_2}\cup
  \dots\cup \overline T_{\ell_s})$
  a path, if $\bar T_{\ell_i}$ and $\bar T_{\ell_{i+1}}$ share a
  common $(d-1)$-dimensional manifold. We will call
  $s_{\ell_1,\ell_s}:=s$ the length of the path $P_{\ell_1,\ell_s}$
  and $\eta=\max_{T\in\tau}\{\text{diam}(T)\}$.
 \end{definition}
\begin{lemma}\label{lem:PF}
   Given $\tau$ from Definition~\ref{def:path} and the index set $\mathcal{J}=\{\ell: \partial T_\ell \cap \gamma \neq \emptyset\}$ it holds
 \begin{equation}
   \CP^2 \lesssim \frac{s_\mathrm{max}r_\mathrm{max}\eta^{d+1}}{|\gamma|H^2},
 \end{equation}
 where $s_\mathrm{max}=\max (s_{k,j})$ is the length of the
 longest path and
 $r_\mathrm{max}=\max_{i\in
   \mathcal{I}}|\{(s,k)\in\mathcal{I}\times\mathcal{J}\mid T_i\in
 P_{k,j}\}|$
 is the maximum number of times the paths intersect.
\end{lemma}
\begin{proof}
  See \cite{PeSc12}.
\end{proof}
We will now use Lemma~\ref{lem:PF} to show some cases when $\CP$ can be
bounded independent of the complex/fine scale boundary
$\partial\Omega$.

\paragraph{Fractal domain.}
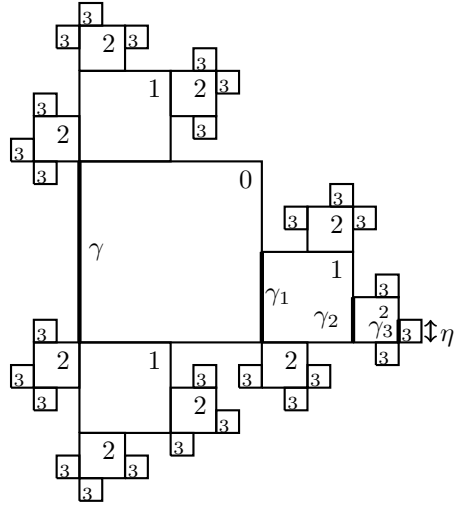
\begin{figure}[htb]
  \centering
\begin{tikzpicture}[scale=0.8]
  \coordinate [label={right:$\gamma$}] (S1) at (0,1.5);
  \coordinate [label={right:$\gamma_1$}] (S1) at (2.9,0.75);
  \coordinate [label={right:$\gamma_2$}] (S1) at (3.7,0.375);
  \coordinate [label={right:$\gamma_3$}] (S1) at (4.6,0.2);

  \newcommand{\sx}{0cm}
  \newcommand{\sy}{0cm}  
  \newcommand{\squarewidth}{3cm}
  \coordinate [label={below right:$ $}] (S1) at (\sx, \sy);
  \coordinate [label={above right:$ $}] (S2) at (\sx, \sy+\squarewidth);
  \draw [ultra thick] (S1) -- (S2);

  \coordinate [label={below right:$ $}] (S1) at (\sx, \sy);
  \coordinate [label={above right:$ $}] (S2) at (\sx, \sy+\squarewidth);
  \coordinate [label={below left:$ 0$}]  (S3) at (\sx+\squarewidth, \sy+\squarewidth);
  \coordinate [label={below left:$ $}]  (S4) at (\sx+\squarewidth, \sy);
  \draw [thick](S1) -- (S2) -- (S3) -- (S4) -- (S1);

  \renewcommand{\sx}{3cm}
  \renewcommand{\sy}{0cm}  
  \renewcommand{\squarewidth}{1.5cm}
  \coordinate [label={below right:$ $}] (S1) at (\sx, \sy);
  \coordinate [label={above right:$ $}] (S2) at (\sx, \sy+\squarewidth);
  \coordinate [label={below left:$ 1$}]  (S3) at (\sx+\squarewidth, \sy+\squarewidth);
  \coordinate [label={below left:$ $}]  (S4) at (\sx+\squarewidth, \sy);
  \draw [thick] (S1) -- (S2) -- (S3) -- (S4) -- (S1);
  \draw [ultra thick] (S1) -- (S2);

  \renewcommand{\sx}{0cm}
  \renewcommand{\sy}{3cm}  
  \renewcommand{\squarewidth}{1.5cm}
  \coordinate [label={below right:$ $}] (S1) at (\sx, \sy);
  \coordinate [label={above right:$ $}] (S2) at (\sx, \sy+\squarewidth);
  \coordinate [label={below left:$ 1$}]  (S3) at (\sx+\squarewidth, \sy+\squarewidth);
  \coordinate [label={below left:$ $}]  (S4) at (\sx+\squarewidth, \sy);
  \draw [thick] (S1) -- (S2) -- (S3) -- (S4) -- (S1);

  \renewcommand{\sx}{0cm}
  \renewcommand{\sy}{-1.5cm}  
  \renewcommand{\squarewidth}{1.5cm}
  \coordinate [label={below right:$ $}] (S1) at (\sx, \sy);
  \coordinate [label={above right:$ $}] (S2) at (\sx, \sy+\squarewidth);
  \coordinate [label={below left:$ 1$}]  (S3) at (\sx+\squarewidth, \sy+\squarewidth);
  \coordinate [label={below left:$ $}]  (S4) at (\sx+\squarewidth, \sy);
  \draw [thick] (S1) -- (S2) -- (S3) -- (S4) -- (S1);

  \renewcommand{\sx}{4.5cm}
  \renewcommand{\sy}{0cm}  
  \renewcommand{\squarewidth}{0.75cm}
  \coordinate [label={below right:$ $}] (S1) at (\sx, \sy);
  \coordinate [label={above right:$ $}] (S2) at (\sx, \sy+\squarewidth);
  \coordinate [label={below left:\scriptsize$ 2$}]  (S3) at (\sx+\squarewidth, \sy+\squarewidth);
  \coordinate [label={below left:$ $}]  (S4) at (\sx+\squarewidth, \sy);
  \draw [thick] (S1) -- (S2) -- (S3) -- (S4) -- (S1);
  \draw [ultra thick] (S1) -- (S2);

  \renewcommand{\sx}{0cm}
  \renewcommand{\sy}{-2.25cm}  
  \renewcommand{\squarewidth}{0.75cm}
  \coordinate [label={below right:$ $}] (S1) at (\sx, \sy);
  \coordinate [label={above right:$ $}] (S2) at (\sx, \sy+\squarewidth);
  \coordinate [label={below left:$ 2$}]  (S3) at (\sx+\squarewidth, \sy+\squarewidth);
  \coordinate [label={below left:$ $}]  (S4) at (\sx+\squarewidth, \sy);
  \draw [thick] (S1) -- (S2) -- (S3) -- (S4) -- (S1);

  \renewcommand{\sx}{0-0.75cm}
  \renewcommand{\sy}{-0.75cm}  
  \renewcommand{\squarewidth}{0.75cm}
  \coordinate [label={below right:$ $}] (S1) at (\sx, \sy);
  \coordinate [label={above right:$ $}] (S2) at (\sx, \sy+\squarewidth);
  \coordinate [label={below left:$ 2$}]  (S3) at (\sx+\squarewidth, \sy+\squarewidth);
  \coordinate [label={below left:$ $}]  (S4) at (\sx+\squarewidth, \sy);
  \draw [thick] (S1) -- (S2) -- (S3) -- (S4) -- (S1);

  \renewcommand{\sx}{1.5cm}
  \renewcommand{\sy}{-1.5cm}  
  \renewcommand{\squarewidth}{0.75cm}
  \coordinate [label={below right:$ $}] (S1) at (\sx, \sy);
  \coordinate [label={above right:$ $}] (S2) at (\sx, \sy+\squarewidth);
  \coordinate [label={below left:$ 2$}]  (S3) at (\sx+\squarewidth, \sy+\squarewidth);
  \coordinate [label={below left:$ $}]  (S4) at (\sx+\squarewidth, \sy);
  \draw [thick] (S1) -- (S2) -- (S3) -- (S4) -- (S1);

  \renewcommand{\sx}{3cm}
  \renewcommand{\sy}{-0.75cm}  
  \renewcommand{\squarewidth}{0.75cm}
  \coordinate [label={below right:$ $}] (S1) at (\sx, \sy);
  \coordinate [label={above right:$ $}] (S2) at (\sx, \sy+\squarewidth);
  \coordinate [label={below left:$ 2$}]  (S3) at (\sx+\squarewidth, \sy+\squarewidth);
  \coordinate [label={below left:$ $}]  (S4) at (\sx+\squarewidth, \sy);
  \draw [thick] (S1) -- (S2) -- (S3) -- (S4) -- (S1);

  \renewcommand{\sx}{3.75cm}
  \renewcommand{\sy}{1.5cm}  
  \renewcommand{\squarewidth}{0.75cm}
  \coordinate [label={below right:$ $}] (S1) at (\sx, \sy);
  \coordinate [label={above right:$ $}] (S2) at (\sx, \sy+\squarewidth);
  \coordinate [label={below left:$ 2$}]  (S3) at (\sx+\squarewidth, \sy+\squarewidth);
  \coordinate [label={below left:$ $}]  (S4) at (\sx+\squarewidth, \sy);
  \draw [thick] (S1) -- (S2) -- (S3) -- (S4) -- (S1);

  \renewcommand{\sx}{0cm}
  \renewcommand{\sy}{4.5cm}  
  \renewcommand{\squarewidth}{0.75cm}
  \coordinate [label={below right:$ $}] (S1) at (\sx, \sy);
  \coordinate [label={above right:$ $}] (S2) at (\sx, \sy+\squarewidth);
  \coordinate [label={below left:$ 2$}]  (S3) at (\sx+\squarewidth, \sy+\squarewidth);
  \coordinate [label={below left:$ $}]  (S4) at (\sx+\squarewidth, \sy);
  \draw [thick] (S1) -- (S2) -- (S3) -- (S4) -- (S1);

  \renewcommand{\sx}{1.5cm}
  \renewcommand{\sy}{3.75cm}  
  \renewcommand{\squarewidth}{0.75cm}
  \coordinate [label={below right:$ $}] (S1) at (\sx, \sy);
  \coordinate [label={above right:$ $}] (S2) at (\sx, \sy+\squarewidth);
  \coordinate [label={below left:$ 2$}]  (S3) at (\sx+\squarewidth, \sy+\squarewidth);
  \coordinate [label={below left:$ $}]  (S4) at (\sx+\squarewidth, \sy);
  \draw [thick] (S1) -- (S2) -- (S3) -- (S4) -- (S1);

  \renewcommand{\sx}{-.75cm}
  \renewcommand{\sy}{3cm}  
  \renewcommand{\squarewidth}{0.75cm}
  \coordinate [label={below right:$ $}] (S1) at (\sx, \sy);
  \coordinate [label={above right:$ $}] (S2) at (\sx, \sy+\squarewidth);
  \coordinate [label={below left:$ 2$}]  (S3) at (\sx+\squarewidth, \sy+\squarewidth);
  \coordinate [label={below left:$ $}]  (S4) at (\sx+\squarewidth, \sy);
  \draw [thick] (S1) -- (S2) -- (S3) -- (S4) -- (S1);

  \renewcommand{\sx}{5.25cm}
  \renewcommand{\sy}{0cm}  
  \renewcommand{\squarewidth}{0.375cm}
  \coordinate [label={below right:$ $}] (S1) at (\sx, \sy);
  \coordinate [label={above right:$ $}] (S2) at (\sx, \sy+\squarewidth);
  \coordinate [label={below left:\scriptsize$ 3$}]  (S3) at (\sx+\squarewidth, \sy+\squarewidth);
  \coordinate [label={below left:$ $}]  (S4) at (\sx+\squarewidth, \sy);
  \draw [thick] (S1) -- (S2) -- (S3) -- (S4) -- (S1);  
  \draw [ultra thick] (S1) -- (S2);

  \coordinate [label={right:$ $}] (S5) at (\sx+\squarewidth+.15cm, \sy);
  \coordinate [label={below right:$\eta $}] (S6) at (\sx+\squarewidth+.15cm, \sy+\squarewidth);
  \draw [thick,<->] (S5) -- (S6);

  \renewcommand{\sx}{4.875cm}
  \renewcommand{\sy}{0.75cm}  
  \renewcommand{\squarewidth}{0.375cm}
  \coordinate [label={below right:$ $}] (S1) at (\sx, \sy);
  \coordinate [label={above right:$ $}] (S2) at (\sx, \sy+\squarewidth);
  \coordinate [label={below left:\scriptsize$ 3$}]  (S3) at (\sx+\squarewidth, \sy+\squarewidth);
  \coordinate [label={below left:$ $}]  (S4) at (\sx+\squarewidth, \sy);
  \draw [thick] (S1) -- (S2) -- (S3) -- (S4) -- (S1);  

  \renewcommand{\sx}{4.875cm}
  \renewcommand{\sy}{-0.375cm}  
  \renewcommand{\squarewidth}{0.375cm}
  \coordinate [label={below right:$ $}] (S1) at (\sx, \sy);
  \coordinate [label={above right:$ $}] (S2) at (\sx, \sy+\squarewidth);
  \coordinate [label={below left:\scriptsize$ 3$}]  (S3) at (\sx+\squarewidth, \sy+\squarewidth);
  \coordinate [label={below left:$ $}]  (S4) at (\sx+\squarewidth, \sy);
  \draw [thick] (S1) -- (S2) -- (S3) -- (S4) -- (S1);  

  \renewcommand{\sx}{0cm}
  \renewcommand{\sy}{5.25cm}  
  \renewcommand{\squarewidth}{0.375cm}
  \coordinate [label={below right:$ $}] (S1) at (\sx, \sy);
  \coordinate [label={above right:$ $}] (S2) at (\sx, \sy+\squarewidth);
  \coordinate [label={below left:\scriptsize$ 3$}]  (S3) at (\sx+\squarewidth, \sy+\squarewidth);
  \coordinate [label={below left:$ $}]  (S4) at (\sx+\squarewidth, \sy);
  \draw [thick] (S1) -- (S2) -- (S3) -- (S4) -- (S1);  

  \renewcommand{\sx}{-0.375cm}
  \renewcommand{\sy}{4.8750cm}  
  \renewcommand{\squarewidth}{0.375cm}
  \coordinate [label={below right:$ $}] (S1) at (\sx, \sy);
  \coordinate [label={above right:$ $}] (S2) at (\sx, \sy+\squarewidth);
  \coordinate [label={below left:\scriptsize$ 3$}]  (S3) at (\sx+\squarewidth, \sy+\squarewidth);
  \coordinate [label={below left:$ $}]  (S4) at (\sx+\squarewidth, \sy);
  \draw [thick] (S1) -- (S2) -- (S3) -- (S4) -- (S1);  

  \renewcommand{\sx}{0.75cm}
  \renewcommand{\sy}{4.8750cm}  
  \renewcommand{\squarewidth}{0.375cm}
  \coordinate [label={below right:$ $}] (S1) at (\sx, \sy);
  \coordinate [label={above right:$ $}] (S2) at (\sx, \sy+\squarewidth);
  \coordinate [label={below left:\scriptsize$ 3$}]  (S3) at (\sx+\squarewidth, \sy+\squarewidth);
  \coordinate [label={below left:$ $}]  (S4) at (\sx+\squarewidth, \sy);
  \draw [thick] (S1) -- (S2) -- (S3) -- (S4) -- (S1);  

  \renewcommand{\sx}{2.25cm}
  \renewcommand{\sy}{4.1250cm}  
  \renewcommand{\squarewidth}{0.375cm}
  \coordinate [label={below right:$ $}] (S1) at (\sx, \sy);
  \coordinate [label={above right:$ $}] (S2) at (\sx, \sy+\squarewidth);
  \coordinate [label={below left:\scriptsize$ 3$}]  (S3) at (\sx+\squarewidth, \sy+\squarewidth);
  \coordinate [label={below left:$ $}]  (S4) at (\sx+\squarewidth, \sy);
  \draw [thick] (S1) -- (S2) -- (S3) -- (S4) -- (S1);  

  \renewcommand{\sx}{1.875cm}
  \renewcommand{\sy}{4.5cm}  
  \renewcommand{\squarewidth}{0.375cm}
  \coordinate [label={below right:$ $}] (S1) at (\sx, \sy);
  \coordinate [label={above right:$ $}] (S2) at (\sx, \sy+\squarewidth);
  \coordinate [label={below left:\scriptsize$ 3$}]  (S3) at (\sx+\squarewidth, \sy+\squarewidth);
  \coordinate [label={below left:$ $}]  (S4) at (\sx+\squarewidth, \sy);
  \draw [thick] (S1) -- (S2) -- (S3) -- (S4) -- (S1);  

  \renewcommand{\sx}{1.875cm}
  \renewcommand{\sy}{3.375cm}  
  \renewcommand{\squarewidth}{0.375cm}
  \coordinate [label={below right:$ $}] (S1) at (\sx, \sy);
  \coordinate [label={above right:$ $}] (S2) at (\sx, \sy+\squarewidth);
  \coordinate [label={below left:\scriptsize$ 3$}]  (S3) at (\sx+\squarewidth, \sy+\squarewidth);
  \coordinate [label={below left:$ $}]  (S4) at (\sx+\squarewidth, \sy);
  \draw [thick] (S1) -- (S2) -- (S3) -- (S4) -- (S1);  

  \renewcommand{\sx}{-1.125cm}
  \renewcommand{\sy}{3cm}  
  \renewcommand{\squarewidth}{0.375cm}
  \coordinate [label={below right:$ $}] (S1) at (\sx, \sy);
  \coordinate [label={above right:$ $}] (S2) at (\sx, \sy+\squarewidth);
  \coordinate [label={below left:\scriptsize$ 3$}]  (S3) at (\sx+\squarewidth, \sy+\squarewidth);
  \coordinate [label={below left:$ $}]  (S4) at (\sx+\squarewidth, \sy);
  \draw [thick] (S1) -- (S2) -- (S3) -- (S4) -- (S1);  

  \renewcommand{\sx}{-0.75cm}
  \renewcommand{\sy}{3.75cm}  
  \renewcommand{\squarewidth}{0.375cm}
  \coordinate [label={below right:$ $}] (S1) at (\sx, \sy);
  \coordinate [label={above right:$ $}] (S2) at (\sx, \sy+\squarewidth);
  \coordinate [label={below left:\scriptsize$ 3$}]  (S3) at (\sx+\squarewidth, \sy+\squarewidth);
  \coordinate [label={below left:$ $}]  (S4) at (\sx+\squarewidth, \sy);
  \draw [thick] (S1) -- (S2) -- (S3) -- (S4) -- (S1);  

  \renewcommand{\sx}{-0.75cm}
  \renewcommand{\sy}{2.6247cm}  
  \renewcommand{\squarewidth}{0.375cm}
  \coordinate [label={below right:$ $}] (S1) at (\sx, \sy);
  \coordinate [label={above right:$ $}] (S2) at (\sx, \sy+\squarewidth);
  \coordinate [label={below left:\scriptsize$ 3$}]  (S3) at (\sx+\squarewidth, \sy+\squarewidth);
  \coordinate [label={below left:$ $}]  (S4) at (\sx+\squarewidth, \sy);
  \draw [thick] (S1) -- (S2) -- (S3) -- (S4) -- (S1);  

  \renewcommand{\sx}{-1.125cm}
  \renewcommand{\sy}{-.75cm}  
  \renewcommand{\squarewidth}{0.375cm}
  \coordinate [label={below right:$ $}] (S1) at (\sx, \sy);
  \coordinate [label={above right:$ $}] (S2) at (\sx, \sy+\squarewidth);
  \coordinate [label={below left:\scriptsize$ 3$}]  (S3) at (\sx+\squarewidth, \sy+\squarewidth);
  \coordinate [label={below left:$ $}]  (S4) at (\sx+\squarewidth, \sy);
  \draw [thick] (S1) -- (S2) -- (S3) -- (S4) -- (S1);  

  \renewcommand{\sx}{-0.75cm}
  \renewcommand{\sy}{0cm}  
  \renewcommand{\squarewidth}{0.375cm}
  \coordinate [label={below right:$ $}] (S1) at (\sx, \sy);
  \coordinate [label={above right:$ $}] (S2) at (\sx, \sy+\squarewidth);
  \coordinate [label={below left:\scriptsize$ 3$}]  (S3) at (\sx+\squarewidth, \sy+\squarewidth);
  \coordinate [label={below left:$ $}]  (S4) at (\sx+\squarewidth, \sy);
  \draw [thick] (S1) -- (S2) -- (S3) -- (S4) -- (S1);  

  \renewcommand{\sx}{-0.75cm}
  \renewcommand{\sy}{-1.1253cm}  
  \renewcommand{\squarewidth}{0.375cm}
  \coordinate [label={below right:$ $}] (S1) at (\sx, \sy);
  \coordinate [label={above right:$ $}] (S2) at (\sx, \sy+\squarewidth);
  \coordinate [label={below left:\scriptsize$ 3$}]  (S3) at (\sx+\squarewidth, \sy+\squarewidth);
  \coordinate [label={below left:$ $}]  (S4) at (\sx+\squarewidth, \sy);
  \draw [thick] (S1) -- (S2) -- (S3) -- (S4) -- (S1);  

  \renewcommand{\sx}{0cm}
  \renewcommand{\sy}{-2.6250cm}  
  \renewcommand{\squarewidth}{0.375cm}
  \coordinate [label={below right:$ $}] (S1) at (\sx, \sy);
  \coordinate [label={above right:$ $}] (S2) at (\sx, \sy+\squarewidth);
  \coordinate [label={below left:\scriptsize$ 3$}]  (S3) at (\sx+\squarewidth, \sy+\squarewidth);
  \coordinate [label={below left:$ $}]  (S4) at (\sx+\squarewidth, \sy);
  \draw [thick] (S1) -- (S2) -- (S3) -- (S4) -- (S1);  

  \renewcommand{\sx}{-0.375cm}
  \renewcommand{\sy}{-2.25cm}  
  \renewcommand{\squarewidth}{0.375cm}
  \coordinate [label={below right:$ $}] (S1) at (\sx, \sy);
  \coordinate [label={above right:$ $}] (S2) at (\sx, \sy+\squarewidth);
  \coordinate [label={below left:\scriptsize$ 3$}]  (S3) at (\sx+\squarewidth, \sy+\squarewidth);
  \coordinate [label={below left:$ $}]  (S4) at (\sx+\squarewidth, \sy);
  \draw [thick] (S1) -- (S2) -- (S3) -- (S4) -- (S1);  

  \renewcommand{\sx}{0.75cm}
  \renewcommand{\sy}{-2.25cm}  
  \renewcommand{\squarewidth}{0.375cm}
  \coordinate [label={below right:$ $}] (S1) at (\sx, \sy);
  \coordinate [label={above right:$ $}] (S2) at (\sx, \sy+\squarewidth);
  \coordinate [label={below left:\scriptsize$ 3$}]  (S3) at (\sx+\squarewidth, \sy+\squarewidth);
  \coordinate [label={below left:$ $}]  (S4) at (\sx+\squarewidth, \sy);
  \draw [thick] (S1) -- (S2) -- (S3) -- (S4) -- (S1);  

  \renewcommand{\sx}{1.5cm}
  \renewcommand{\sy}{-1.8750cm}  
  \renewcommand{\squarewidth}{0.375cm}
  \coordinate [label={below right:$ $}] (S1) at (\sx, \sy);
  \coordinate [label={above right:$ $}] (S2) at (\sx, \sy+\squarewidth);
  \coordinate [label={below left:\scriptsize$ 3$}]  (S3) at (\sx+\squarewidth, \sy+\squarewidth);
  \coordinate [label={below left:$ $}]  (S4) at (\sx+\squarewidth, \sy);
  \draw [thick] (S1) -- (S2) -- (S3) -- (S4) -- (S1);  

  \renewcommand{\sx}{2.25cm}
  \renewcommand{\sy}{-1.5cm}  
  \renewcommand{\squarewidth}{0.375cm}
  \coordinate [label={below right:$ $}] (S1) at (\sx, \sy);
  \coordinate [label={above right:$ $}] (S2) at (\sx, \sy+\squarewidth);
  \coordinate [label={below left:\scriptsize$ 3$}]  (S3) at (\sx+\squarewidth, \sy+\squarewidth);
  \coordinate [label={below left:$ $}]  (S4) at (\sx+\squarewidth, \sy);
  \draw [thick] (S1) -- (S2) -- (S3) -- (S4) -- (S1);  

  \renewcommand{\sx}{1.875cm}
  \renewcommand{\sy}{-0.75cm}  
  \renewcommand{\squarewidth}{0.375cm}
  \coordinate [label={below right:$ $}] (S1) at (\sx, \sy);
  \coordinate [label={above right:$ $}] (S2) at (\sx, \sy+\squarewidth);
  \coordinate [label={below left:\scriptsize$ 3$}]  (S3) at (\sx+\squarewidth, \sy+\squarewidth);
  \coordinate [label={below left:$ $}]  (S4) at (\sx+\squarewidth, \sy);
  \draw [thick] (S1) -- (S2) -- (S3) -- (S4) -- (S1);  

  \renewcommand{\sx}{2.6250cm}
  \renewcommand{\sy}{-0.75cm}  
  \renewcommand{\squarewidth}{0.375cm}
  \coordinate [label={below right:$ $}] (S1) at (\sx, \sy);
  \coordinate [label={above right:$ $}] (S2) at (\sx, \sy+\squarewidth);
  \coordinate [label={below left:\scriptsize$ 3$}]  (S3) at (\sx+\squarewidth, \sy+\squarewidth);
  \coordinate [label={below left:$ $}]  (S4) at (\sx+\squarewidth, \sy);
  \draw [thick] (S1) -- (S2) -- (S3) -- (S4) -- (S1);  

  \renewcommand{\sx}{3.75cm}
  \renewcommand{\sy}{-0.75cm}  
  \renewcommand{\squarewidth}{0.375cm}
  \coordinate [label={below right:$ $}] (S1) at (\sx, \sy);
  \coordinate [label={above right:$ $}] (S2) at (\sx, \sy+\squarewidth);
  \coordinate [label={below left:\scriptsize$ 3$}]  (S3) at (\sx+\squarewidth, \sy+\squarewidth);
  \coordinate [label={below left:$ $}]  (S4) at (\sx+\squarewidth, \sy);
  \draw [thick] (S1) -- (S2) -- (S3) -- (S4) -- (S1);  

  \renewcommand{\sx}{3.3750cm}
  \renewcommand{\sy}{-1.125cm}  
  \renewcommand{\squarewidth}{0.375cm}
  \coordinate [label={below right:$ $}] (S1) at (\sx, \sy);
  \coordinate [label={above right:$ $}] (S2) at (\sx, \sy+\squarewidth);
  \coordinate [label={below left:\scriptsize$ 3$}]  (S3) at (\sx+\squarewidth, \sy+\squarewidth);
  \coordinate [label={below left:$ $}]  (S4) at (\sx+\squarewidth, \sy);
  \draw [thick] (S1) -- (S2) -- (S3) -- (S4) -- (S1);  

  \renewcommand{\sx}{4.1250cm}
  \renewcommand{\sy}{2.25cm}  
  \renewcommand{\squarewidth}{0.375cm}
  \coordinate [label={below right:$ $}] (S1) at (\sx, \sy);
  \coordinate [label={above right:$ $}] (S2) at (\sx, \sy+\squarewidth);
  \coordinate [label={below left:\scriptsize$ 3$}]  (S3) at (\sx+\squarewidth, \sy+\squarewidth);
  \coordinate [label={below left:$ $}]  (S4) at (\sx+\squarewidth, \sy);
  \draw [thick] (S1) -- (S2) -- (S3) -- (S4) -- (S1);  

  \renewcommand{\sx}{4.5cm}
  \renewcommand{\sy}{1.875cm}  
  \renewcommand{\squarewidth}{0.375cm}
  \coordinate [label={below right:$ $}] (S1) at (\sx, \sy);
  \coordinate [label={above right:$ $}] (S2) at (\sx, \sy+\squarewidth);
  \coordinate [label={below left:\scriptsize$ 3$}]  (S3) at (\sx+\squarewidth, \sy+\squarewidth);
  \coordinate [label={below left:$ $}]  (S4) at (\sx+\squarewidth, \sy);
  \draw [thick] (S1) -- (S2) -- (S3) -- (S4) -- (S1);  

  \renewcommand{\sx}{3.3750cm}
  \renewcommand{\sy}{1.875cm}  
  \renewcommand{\squarewidth}{0.375cm}
  \coordinate [label={below right:$ $}] (S1) at (\sx, \sy);
  \coordinate [label={above right:$ $}] (S2) at (\sx, \sy+\squarewidth);
  \coordinate [label={below left:\scriptsize$ 3$}]  (S3) at (\sx+\squarewidth, \sy+\squarewidth);
  \coordinate [label={below left:$ $}]  (S4) at (\sx+\squarewidth, \sy);
  \draw [thick] (S1) -- (S2) -- (S3) -- (S4) -- (S1);
\end{tikzpicture}
  \caption{A fractal domain that has a bounded Poincar\'e-Friedrichs constant.}
  \label{fig:fractal}
\end{figure}
We consider the fractal shaped domain given in
Figure~\ref{fig:fractal}. First we compute $s_\mathrm{max}$.  The
number of $T_\ell$ on $\gamma$ is then proportional to $2^k$, where $k$
is the total number of uniform refinements of the domain and we bound
the maximum path length as
\begin{equation}
  s_\mathrm{max}\sim\sum_{i=0}^{k}\frac{2^k}{2^i} \leq 2\cdot 2^{k},
\end{equation}
i.e., the maximum length of a path is proportional to $2^k$. Next we compute
the maximum number of times a simplex is in a path, $r_\mathrm{max}$. First we
show how many times the elements besides $\gamma$ are in a path and then we
show that this number is larger than on any other $\gamma_i$, see
Figure~\ref{fig:fractal}. The number of paths on each $T_\ell$ is the total
number of elements in the domain. On $\gamma$ we get
\begin{equation}
  r_\gamma \sim \sum^{k}_{i=0} n_\square(i)e_\square(i) = \sum^{k}_{i=0} 3^i \frac{(2^{k})^2}{4^{i}}\leq 4^{k}\sum^{k}_{i=0}  \left(\frac{3}{4}\right)^i \leq 4\cdot4^{k},
\end{equation}
where $n_\square(i)$ is the number sub domains with index $i$ and
$e_\square(i)$ is the number of elements inside a single sub domain
with index $i$. Next we show that there is no $T$ in other parts of
the domain where the number of paths is proportional to something
with a stronger dependence on $n$ than $r_\gamma$. We obtain,
\begin{equation}
  \begin{aligned}
    r_{\gamma^j} &\sim \frac{2^j}{3^j}\sum^{k}_{i=j} n_\square(i)e_\square(i) < r_\gamma,
  \end{aligned}
\end{equation}
where $2^j$ comes from that $\gamma^j$ is $2^j$ times smaller than $\gamma$
and $1/3^j$ since the area of the domain affecting boundary $r_{\gamma^j}$ is
less than $3^j$ of the total area. This proves that $r_\mathrm{max}\sim
r_\gamma$, choosing structured paths in the interior of the squares. To finish
the argument we note that $H/\eta = 2^k$ and $|\gamma|=H$, and we obtain
\begin{equation}
  \CP^2 \lesssim \frac{s_\mathrm{max}r_\mathrm{max}\eta^{d+1}}{|\gamma|H^2} \lesssim  1.
\end{equation}

\paragraph{Saw tooth domain.}
An other example of a complex geometry is the saw domain given in
Figure~\ref{fig:saw}.
\begin{figure}[htb]
  \centering
\begin{tikzpicture}[scale=0.8]
  \coordinate [label={right:$\gamma$}] (S1) at (0,1.5);

  \newcommand{\sx}{0cm}
  \newcommand{\sy}{0cm}  
  \newcommand{\squarewidth}{3cm}
  \coordinate [label={below right:$ $}] (S1) at (\sx, \sy);
  \coordinate [label={above right:$ $}] (S2) at (\sx, \sy+\squarewidth);
  \draw [ultra thick] (S1) -- (S2);

  \coordinate [label={below right:$ $}] (S1) at (\sx, \sy);
  \coordinate [label={above right:$ $}] (S2) at (\sx, \sy+\squarewidth);
  \coordinate [label={below left:$ $}]  (S3) at (\sx+\squarewidth, \sy+\squarewidth);
  \coordinate [label={below left:$ $}]  (S4) at (\sx+\squarewidth, \sy);
  \draw [thick] (S1) -- (S2) -- (S3);
  \draw [thick] (S4) -- (S1);

  \renewcommand{\sx}{3cm}
  \renewcommand{\sy}{0cm}  
  \newcommand{\len}{1cm}
  \newcommand{\wid}{0.1cm}
  \coordinate [label={below right:$ $}] (S1) at (\sx, \sy);
  \coordinate [label={above right:$ $}] (S2) at (\sx, \sy+\wid);
  \coordinate [label={below left:$ $}]  (S3) at (\sx+\len, \sy+\wid);
  \coordinate [label={below left:$ $}]  (S4) at (\sx+\len, \sy);
  \draw [thick]  (S2) -- (S3) --  (S4) -- (S1);

  \renewcommand{\sx}{3cm}
  \renewcommand{\sy}{0.2cm}  
  \coordinate [label={below right:$ $}] (S1) at (\sx, \sy);
  \coordinate [label={above right:$ $}] (S2) at (\sx, \sy+\wid);
  \coordinate [label={below left:$ $}]  (S3) at (\sx+\len, \sy+\wid);
  \coordinate [label={below left:$ $}]  (S4) at (\sx+\len, \sy);
  \draw [thick]  (S2) -- (S3) --  (S4) -- (S1);

  \renewcommand{\sx}{3cm}
  \renewcommand{\sy}{0.4cm}  
  \coordinate [label={below right:$ $}] (S1) at (\sx, \sy);
  \coordinate [label={above right:$ $}] (S2) at (\sx, \sy+\wid);
  \coordinate [label={below left:$ $}]  (S3) at (\sx+\len, \sy+\wid);
  \coordinate [label={below left:$ $}]  (S4) at (\sx+\len, \sy);
  \draw [thick]  (S2) -- (S3) --  (S4) -- (S1);

  \renewcommand{\sx}{3cm}
  \renewcommand{\sy}{0.6cm}  
  \coordinate [label={below right:$ $}] (S1) at (\sx, \sy);
  \coordinate [label={above right:$ $}] (S2) at (\sx, \sy+\wid);
  \coordinate [label={below left:$ $}]  (S3) at (\sx+\len, \sy+\wid);
  \coordinate [label={below left:$ $}]  (S4) at (\sx+\len, \sy);
  \draw [thick]  (S2) -- (S3) --  (S4) -- (S1);

  \renewcommand{\sx}{3cm}
  \renewcommand{\sy}{0.8cm}  
  \coordinate [label={below right:$ $}] (S1) at (\sx, \sy);
  \coordinate [label={above right:$ $}] (S2) at (\sx, \sy+\wid);
  \coordinate [label={below left:$ $}]  (S3) at (\sx+\len, \sy+\wid);
  \coordinate [label={below left:$ $}]  (S4) at (\sx+\len, \sy);
  \draw [thick]  (S2) -- (S3) --  (S4) -- (S1);

  \renewcommand{\sx}{3cm}
  \renewcommand{\sy}{1.0cm}  
  \coordinate [label={below right:$ $}] (S1) at (\sx, \sy);
  \coordinate [label={above right:$ $}] (S2) at (\sx, \sy+\wid);
  \coordinate [label={below left:$ $}]  (S3) at (\sx+\len, \sy+\wid);
  \coordinate [label={below left:$ $}]  (S4) at (\sx+\len, \sy);
  \draw [thick]  (S2) -- (S3) --  (S4) -- (S1);

  \renewcommand{\sx}{3cm}
  \renewcommand{\sy}{1.2cm}  
  \coordinate [label={below right:$ $}] (S1) at (\sx, \sy);
  \coordinate [label={above right:$ $}] (S2) at (\sx, \sy+\wid);
  \coordinate [label={below left:$ $}]  (S3) at (\sx+\len, \sy+\wid);
  \coordinate [label={below left:$ $}]  (S4) at (\sx+\len, \sy);
  \draw [thick]  (S2) -- (S3) --  (S4) -- (S1);
  \renewcommand{\sx}{3cm}
  \renewcommand{\sy}{1.4cm}  
  \coordinate [label={below right:$ $}] (S1) at (\sx, \sy);
  \coordinate [label={above right:$ $}] (S2) at (\sx, \sy+\wid);
  \coordinate [label={below left:$ $}]  (S3) at (\sx+\len, \sy+\wid);
  \coordinate [label={below left:$ $}]  (S4) at (\sx+\len, \sy);
  \draw [thick]  (S2) -- (S3) --  (S4) -- (S1);
  \renewcommand{\sx}{3cm}
  \renewcommand{\sy}{1.6cm}  
  \coordinate [label={below right:$ $}] (S1) at (\sx, \sy);
  \coordinate [label={above right:$ $}] (S2) at (\sx, \sy+\wid);
  \coordinate [label={below left:$ $}]  (S3) at (\sx+\len, \sy+\wid);
  \coordinate [label={below left:$ $}]  (S4) at (\sx+\len, \sy);
  \draw [thick]  (S2) -- (S3) --  (S4) -- (S1);
  \renewcommand{\sx}{3cm}
  \renewcommand{\sy}{1.8cm}  
  \coordinate [label={below right:$ $}] (S1) at (\sx, \sy);
  \coordinate [label={above right:$ $}] (S2) at (\sx, \sy+\wid);
  \coordinate [label={below left:$ $}]  (S3) at (\sx+\len, \sy+\wid);
  \coordinate [label={below left:$ $}]  (S4) at (\sx+\len, \sy);
  \draw [thick]  (S2) -- (S3) --  (S4) -- (S1);
  \renewcommand{\sx}{3cm}
  \renewcommand{\sy}{2.0cm}  
  \coordinate [label={below right:$ $}] (S1) at (\sx, \sy);
  \coordinate [label={above right:$ $}] (S2) at (\sx, \sy+\wid);
  \coordinate [label={below left:$ $}]  (S3) at (\sx+\len, \sy+\wid);
  \coordinate [label={below left:$ $}]  (S4) at (\sx+\len, \sy);
  \draw [thick]  (S2) -- (S3) --  (S4) -- (S1);
  \renewcommand{\sx}{3cm}
  \renewcommand{\sy}{2.2cm}  
  \coordinate [label={below right:$ $}] (S1) at (\sx, \sy);
  \coordinate [label={above right:$ $}] (S2) at (\sx, \sy+\wid);
  \coordinate [label={below left:$ $}]  (S3) at (\sx+\len, \sy+\wid);
  \coordinate [label={below left:$ $}]  (S4) at (\sx+\len, \sy);
  \draw [thick]  (S2) -- (S3) --  (S4) -- (S1);

  \renewcommand{\sx}{3cm}
  \renewcommand{\sy}{2.4cm}  
  \coordinate [label={below right:$ $}] (S1) at (\sx, \sy);
  \coordinate [label={above right:$ $}] (S2) at (\sx, \sy+\wid);
  \coordinate [label={below left:$ $}]  (S3) at (\sx+\len, \sy+\wid);
  \coordinate [label={below left:$ $}]  (S4) at (\sx+\len, \sy);
  \draw [thick]  (S2) -- (S3) --  (S4) -- (S1);

  \renewcommand{\sx}{3cm}
  \renewcommand{\sy}{2.6cm}  
  \coordinate [label={below right:$ $}] (S1) at (\sx, \sy);
  \coordinate [label={above right:$ $}] (S2) at (\sx, \sy+\wid);
  \coordinate [label={below left:$ $}]  (S3) at (\sx+\len, \sy+\wid);
  \coordinate [label={below left:$ $}]  (S4) at (\sx+\len, \sy);
  \draw [thick]  (S2) -- (S3) --  (S4) -- (S1);

  \renewcommand{\sx}{3cm}
  \renewcommand{\sy}{2.8cm}  
  \coordinate [label={below right:$ $}] (S1) at (\sx, \sy);
  \coordinate [label={above right:$ $}] (S2) at (\sx, \sy+\wid);
  \coordinate [label={below left:$ $}]  (S3) at (\sx+\len, \sy+\wid);
  \coordinate [label={below left:$ $}]  (S4) at (\sx+\len, \sy);
  \draw [thick]  (S2) -- (S3) --  (S4) -- (S1);

  \renewcommand{\sx}{3cm}
  \renewcommand{\sy}{0.1cm}  
  \coordinate [label={below right:$ $}] (S1) at (\sx, \sy);
  \coordinate [label={above right:$ $}] (S2) at (\sx, \sy+\wid);
  \draw [thick]  (S2) -- (S1);

  \renewcommand{\sx}{3cm}
  \renewcommand{\sy}{0.3cm}  
  \coordinate [label={below right:$ $}] (S1) at (\sx, \sy);
  \coordinate [label={above right:$ $}] (S2) at (\sx, \sy+\wid);
  \draw [thick]  (S2) -- (S1);
  \renewcommand{\sx}{3cm}
  \renewcommand{\sy}{0.5cm}  
  \coordinate [label={below right:$ $}] (S1) at (\sx, \sy);
  \coordinate [label={above right:$ $}] (S2) at (\sx, \sy+\wid);
  \draw [thick]  (S2) -- (S1);
  \renewcommand{\sx}{3cm}
  \renewcommand{\sy}{0.7cm}  
  \coordinate [label={below right:$ $}] (S1) at (\sx, \sy);
  \coordinate [label={above right:$ $}] (S2) at (\sx, \sy+\wid);
  \draw [thick]  (S2) -- (S1);
  \renewcommand{\sx}{3cm}
  \renewcommand{\sy}{0.9cm}  
  \coordinate [label={below right:$ $}] (S1) at (\sx, \sy);
  \coordinate [label={above right:$ $}] (S2) at (\sx, \sy+\wid);
  \draw [thick]  (S2) -- (S1);
  \renewcommand{\sx}{3cm}
  \renewcommand{\sy}{1.1cm}  
  \coordinate [label={below right:$ $}] (S1) at (\sx, \sy);
  \coordinate [label={above right:$ $}] (S2) at (\sx, \sy+\wid);
  \draw [thick]  (S2) -- (S1);
  \renewcommand{\sx}{3cm}
  \renewcommand{\sy}{1.3cm}  
  \coordinate [label={below right:$ $}] (S1) at (\sx, \sy);
  \coordinate [label={above right:$ $}] (S2) at (\sx, \sy+\wid);
  \draw [thick]  (S2) -- (S1);
  \renewcommand{\sx}{3cm}
  \renewcommand{\sy}{1.5cm}  
  \coordinate [label={below right:$ $}] (S1) at (\sx, \sy);
  \coordinate [label={above right:$ $}] (S2) at (\sx, \sy+\wid);
  \draw [thick]  (S2) -- (S1);
  \renewcommand{\sx}{3cm}
  \renewcommand{\sy}{1.7cm}  
  \coordinate [label={below right:$ $}] (S1) at (\sx, \sy);
  \coordinate [label={above right:$ $}] (S2) at (\sx, \sy+\wid);
  \draw [thick]  (S2) -- (S1);
  \renewcommand{\sx}{3cm}
  \renewcommand{\sy}{1.9cm}  
  \coordinate [label={below right:$ $}] (S1) at (\sx, \sy);
  \coordinate [label={above right:$ $}] (S2) at (\sx, \sy+\wid);
  \draw [thick]  (S2) -- (S1);
  \renewcommand{\sx}{3cm}
  \renewcommand{\sy}{2.1cm}  
  \coordinate [label={below right:$ $}] (S1) at (\sx, \sy);
  \coordinate [label={above right:$ $}] (S2) at (\sx, \sy+\wid);
  \draw [thick]  (S2) -- (S1);
  \renewcommand{\sx}{3cm}
  \renewcommand{\sy}{2.3cm}  
  \coordinate [label={below right:$ $}] (S1) at (\sx, \sy);
  \coordinate [label={above right:$ $}] (S2) at (\sx, \sy+\wid);
  \draw [thick]  (S2) -- (S1);
  \renewcommand{\sx}{3cm}
  \renewcommand{\sy}{2.5cm}  
  \coordinate [label={below right:$ $}] (S1) at (\sx, \sy);
  \coordinate [label={above right:$ $}] (S2) at (\sx, \sy+\wid);
  \draw [thick]  (S2) -- (S1);
  \renewcommand{\sx}{3cm}
  \renewcommand{\sy}{2.7cm}  
  \coordinate [label={below right:$ $}] (S1) at (\sx, \sy);
  \coordinate [label={above right:$ $}] (S2) at (\sx, \sy+\wid);
  \draw [thick]  (S2) -- (S1);
  \renewcommand{\sx}{3cm}
  \renewcommand{\sy}{2.9cm}  
  \coordinate [label={below right:$ $}] (S1) at (\sx, \sy);
  \coordinate [label={above right:$ $}] (S2) at (\sx, \sy+\wid);
  \draw [thick]  (S2) -- (S1);
\end{tikzpicture}
  \caption{Saw tooth domain that has a bounded Poincar\'e-Friedrichs constant. Here $\eta$ is the width of one of the saw teeth.}
  \label{fig:saw}
\end{figure}
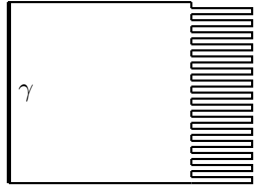
Let the width of the saw teeth be $\eta = 2^{-k}$. A mesh constructed using
$2^k$ uniform refinements are needed to resolve the saw teeth. It is clear
that $s_\mathrm{max}\sim 2^k$ and choosing the structural paths we have that
$r_\mathrm{max}\sim (2^k)^2$. Again we have that $\CP \lesssim 1$ as long the
length of the saw teeth are fixed.

An example of a domain with a non-bounded Poincar\'e-Friedrichs
constant is e.g. a dumbbell domain.

\subsection{Estimation of the interpolation error}
In this section we compute the interpolation error for a class of fine
scale functions needed in the analysis.  For each $T\in\mathcal{T}_H$
we define an $L$-layer element patch recursively as
\begin{equation}
  \begin{aligned}
    \omega_T^0 &:= T\cap\Omega, \\
    \omega_T^\ell &:= \text{int}\left((\bar T \in\T_H\mid\bar T\cap \overline\omega_T^{\ell-1}\neq 0)\cap\Omega\right),\quad\text{for }\ell= 1,\dots,L.
  \end{aligned}
\end{equation}
\begin{lemma}
  The projective Cl\'{e}ment type operator inherits the local approximation and stability properties for all interior elements, i.e., for all $v\in H^1(\Omega)$
  \begin{equation}
    \|H^{-1}(v-\mathcal{I}_Hv) \|_{T} + \|\nabla \mathcal{I}_Hv\|_{T} \lesssim \|\nabla v\|_{\omega_T^1},
  \end{equation}
  holds for all interior elements $T\in\T_H$.
\end{lemma}
\begin{proof}
  It follows directly from the standard proof \cite{BrPe14} since
  $\sum_{i\in\N}\varphi_i$ is a partition of unity on interior elements.
\end{proof}

The trace of a function $v\in\V^f$ is ``small'' since the function $v$
is in the kernel of an averaging operator,
\begin{equation}
  \V^f =\{v\in \VGammah \mid \mathcal{I}_H v = 0\}.
\end{equation}
We formulate this more precisely in the following Lemma.

\begin{lemma}\label{lem:trace}
  Given an interior element $T\in\T_H$ let $\gamma\subset\partial T$
  be one of its faces. Then
  \begin{equation}
    \|v\|_{\gamma} \lesssim H^{1/2} \|\nabla v \|_{\omega_T^1},
  \end{equation}
  holds for all $v\in\V^f$.
\end{lemma}
\begin{proof}
  For an interior element $T$ the standard approximation
  property of the Clem\'ent type interpolation operator holds, i.e.,
  \begin{equation}\label{eq:interiorestimate}
    \|v\|_{T} = \|v-\mathcal{I}_Hv\|_{T}\lesssim H\|\nabla v\|_{\omega_T^1}.
  \end{equation}
  since $v\in\V^f$. Using a trace inequality and \eqref{eq:interiorestimate} we obtain
  \begin{equation}
    \begin{aligned}
      \|v\|^2_{\gamma} &\lesssim H^{-1} \|v\|^2_{T} + H \|\nabla v\|^2_{T}\lesssim H \|\nabla v\|^2_{\omega_T^1},
    \end{aligned}
  \end{equation}
  where $T$ is an interior element.
\end{proof}

We will now make an assumption which is a sufficient condition to prove
the main results of the paper and which will also simplifies the analysis.
\begin{assumption}\label{ass:interior_element}
  All elements in $S\in\mathcal{T}_H$ share a vertex with an {\em
    interior element}, i.e., an element $T\in\T_H$ such that
  $T\cap\Omega=T$.
\end{assumption}
\begin{figure}[htb]
  \centering
  \includegraphics[width=0.6\textwidth]{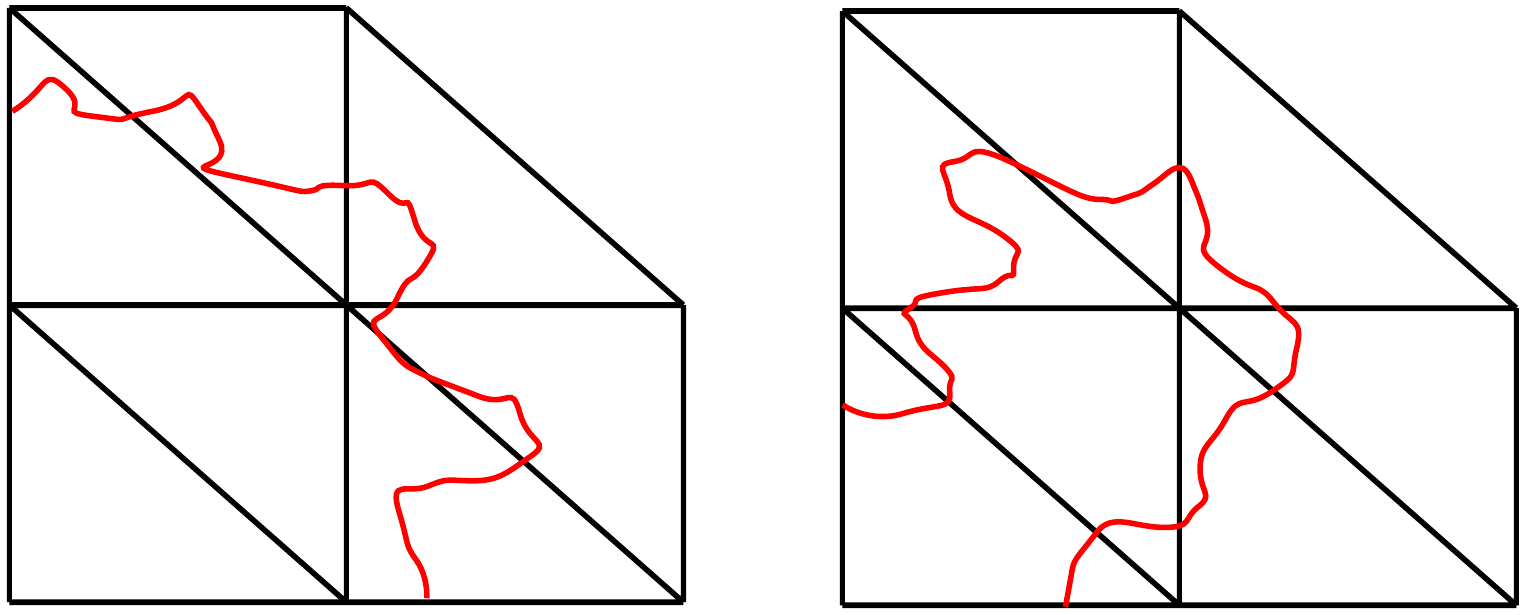}
  \caption{Admissible (left) and non-admissible (right) mesh according
    to Assumption~\ref{ass:interior_element}. The irregular curve is where the
    elements are cut by the outer boundary.}
  \label{fig:cut_elements}
\end{figure}
\begin{lemma}\label{lem:interiorestimate}
  Let $T$ be an element that is cut by the boundary
  $\partial\Omega$. Under Assumption~\ref{ass:interior_element} the
  following Poincar\'e-Friedrichs type inequality holds
  \begin{equation}
    \|v\|_{T}\lesssim  H\|\nabla v\|_{\omega^2_{T}}\quad\text{for all }v\in\V^f.
  \end{equation}
\end{lemma}
\begin{proof}
  Let $\widetilde T$ be an element which share the vertex $x$ with
  $T$. Using \eqref{eq:general_PF} we have that
  \begin{equation}
    \begin{aligned}
      \|v\|_T \leq \|v\|_{\omega_x^0}\lesssim H\|\nabla v\|_{\omega_x^0} + H^{1/2}\|v\|_{\gamma} \lesssim H\|\nabla v\|_{\omega^1_{\widetilde T}}\lesssim H\|\nabla v\|_{\omega^2_{ T}},
    \end{aligned}
  \end{equation}
  holds, since $\text{diam}(\omega^0_{x})\lesssim H$,
  $\text{diam}(\omega^1_{\widetilde T})\lesssim H$ and $\omega_x^0\subset
  \omega^1_{\widetilde T}\subset \omega^2_{T}$ choosing $\overline\gamma$ as an
  edge of $\widetilde T$ which contains the vertex $x$.
\end{proof}

Next we prove local approximation and stability properties
for functions which are in some sense close to but outside $\V^f$.
\begin{lemma}\label{lem:approxIH}
  Let $\mathcal{I}_H:L^2(\Omega_H)\to\V_H$ be the Cl\'ement
  interpolation operator defined by \eqref{eq:clement}. If $v=\eta w$,
  where $\eta$ and $w$ satisfies $0\leq \eta\leq 1$,
  $\|\nabla\eta\|_{L^{\infty}}\lesssim H^{-1}$, and $w\in\V^f$, then
  the following estimate holds
  \begin{equation}
    \|H^{-1}(v-\mathcal{I}_Hv) \|_{T} + \|\nabla \mathcal{I}_Hv\|_{T} \lesssim \|\nabla w\|_{\omega_T^2},
  \end{equation}
  for all $T\in\T_H$.
\end{lemma}
\begin{proof}
  The local approximation and stability properties for an interior
  element follows directly from Lemma~\ref{lem:interiorestimate}
  together with
  \begin{equation}
    \|\nabla \eta w\|_{T} \lesssim  H^{-1}\| w\|_{T} + \|\nabla w\|_{T} \lesssim \|\nabla w\|_{\omega_T^1}.
  \end{equation}
  Next we investigate the local approximation and stability properties
  for elements on the boundary. Let $T$ be an element cut by the
  boundary and $\widetilde T$ an interior element sharing vertex $x$
  with T. Then the $L^2$-stability follows directly from the stability
  of the interior elements, i.e,
  \begin{equation}
    \begin{aligned}
    \|( P_xv)(x)\|_{T} &= |(P_xv)(x)|\frac{\|1\|_{T}}{\|1\|_{\widetilde T}}\|1\|_{\widetilde T} \lesssim |(P_xv)(x)|\|1\|_{\widetilde T}=\|(P_xv)(x)\|_{\widetilde T} \\
    & \lesssim H^{d/2}\|(P_xv)(x)\|_{L^\infty(\widetilde T)} \leq H^{d/2}\|P_xv\|_{L^\infty(\widetilde T)} \lesssim \|P_xv\|_{\widetilde T} \\
    & \leq \|P_xv\|_{\omega^0_x}\leq \|v\|_{\omega^0_x}\leq \|v\|_{\omega^1_T}.
    \end{aligned}
  \end{equation}
  We obtain
  \begin{equation}
    \begin{aligned}
      \|v-\mathcal{I}_Hv\|_{T} &\lesssim
      \|v\|_{\omega^1_T}\lesssim \|w\|_{\omega^1_T} \lesssim H\|\nabla
      w\|_{\omega^2_T},
    \end{aligned}
  \end{equation}
  since $w\in\V^f$ using Lemma~\ref{lem:trace}, $L^2$-stability of the
  interpolation operator, and $w\in\V^f$. Similar argument yield
  \begin{equation}
    \begin{aligned}
      \|\nabla \mathcal{I}_H v\|_{T}&\lesssim \sum_{x\in\mathcal{N}_T}
      H^{-1}\|(P_xv)(x)\|_{T} \lesssim \|\nabla w
      \|_{\omega_T^2},
    \end{aligned}
  \end{equation}
  where $\mathcal{N}_T$ is all vertices in element $T$.
\end{proof}

\subsection{Estimation of the error in the reference finite element solution}
\label{sec:estaimte_uref}

We bound the error in the reference finite element solution $\uref\in\VGammah$.
\begin{lemma}\label{lemma:H1enrichment}
  Let $u\in \V$ and $\uref\in\VGammah$ be the solutions to
  \eqref{eq:week} and \eqref{eq:omegagamma} respectively. Then
  \begin{equation}
    |||u-\uref||| \lesssim  \inf_{v_H\in \V_H}\left(\|H^{-1}(u-v_H)\|_{\Omega\setminus \omega_{\Gamma}^{k-1}}+|||u - v_H|||_{\Omega\setminus \omega_{\Gamma}^{k-1}}\right) + \inf_{v_h\in \V_h(\omega_{\Gamma}^k)}|||u-v_h|||_{\omega_{\Gamma}^{k}},
  \end{equation}
  holds.
\end{lemma}
\begin{proof}
  We split $\Omega$ into the different parts
  $\Omega\setminus\omega_\Gamma^k$,
  $\omega_{\Gamma}^{k}\setminus\omega_{\Gamma}^{k-1}$, and
  $\omega_{\Gamma}^{k-1}$. Since $\VGammah\subset \V$, we have the best
approximation result
\begin{equation}\label{eq:best_approximation_referense}
  |||u - \uref ||| \lesssim |||u - w|||,\qquad\text{for all }w\in\VGammah.
\end{equation}
Let $\eta\in V_H$ be a cut off function, where
$\eta|_{\Omega\setminus\omega_\Gamma^k}=0$,
$\eta|_{\omega_\Gamma^{k-1}}=1$, and
$\|\nabla\eta\|_{L^{\infty}(T)}\lesssim H^{-1}$. We construct
$w=v_H + \pi_h\eta(v_h-v_H)\in \VGammah$ where $v_H\in \V_H$,
$v_h\in\V_h(\omega_\Gamma^k)$ , and $\pi_h$ is the nodal interpolant
onto the finite element space $\V_h$ and obtain
\begin{equation}
  \begin{aligned}
  |||u - w|||^2 = |||u - v_H|||^2_{\Omega\setminus\omega_\Gamma^k} + |||u - v_H - \pi_h\eta(v_h-v_H)|||^2_{\omega_\Gamma^k\setminus\omega_\Gamma^{k-1}}
  + |||u - v_h|||^2_{\omega_\Gamma^{k-1}} .
  \end{aligned}
\end{equation}
The first and third term are in the right form, see the statement of
Lemma~\ref{lemma:H1enrichment}. We turn to the second term. Using
the fact that the nodal interpolant $\pi_h$ is $H^1$-stable for finite
polynomial degrees (2 in our case) we obtain
\begin{equation}
  \begin{aligned}
    |||\pi_h\eta(v_h-v_H)|||_{\omega_\Gamma^k\setminus\omega_\Gamma^{k-1}}^2&\lesssim |||\eta(v_h- v_H)|||_{\omega_\Gamma^k\setminus\omega_\Gamma^{k-1}}^2 \\
    & = ||\nabla(\eta(v_h- v_H))||^2_{\omega_\Gamma^k\setminus\omega_\Gamma^{k-1}}+\kappa||\eta(v_h- v_H)||^2_{\partial(\omega_\Gamma^k\setminus\omega_\Gamma^{k-1})\cap\partial\Gamma_R}. \\
  \end{aligned}
\end{equation}
We have
\begin{equation}
  \begin{aligned}
    &||\nabla(\eta(v_h- v_H))||_{\omega_\Gamma^k\setminus\omega_\Gamma^{k-1}} \\
    &\qquad\leq ||(v_h- v_H)\nabla\eta||_{\omega_\Gamma^k\setminus\omega_\Gamma^{k-1}}+||\eta\nabla(v_h- v_H)||_{\omega_\Gamma^k\setminus\omega_\Gamma^{k-1}} \\
    &\qquad \lesssim H^{-1}||v_h- v_H||_{\omega_\Gamma^k\setminus\omega_\Gamma^{k-1}}+||\nabla(v_h- v_H)||_{\omega_\Gamma^k\setminus\omega_\Gamma^{k-1}} \\
    &\qquad \lesssim H^{-1}||v_h-u+u- v_H||_{\omega_\Gamma^k\setminus\omega_\Gamma^{k-1}}+||\nabla(v_h-u+u-v_H)||_{\omega_\Gamma^k\setminus\omega_\Gamma^{k-1}} \\
    &\qquad \lesssim H^{-1}||u- v_H||_{\omega_\Gamma^k\setminus\omega_\Gamma^{k-1}}+||\nabla(u-v_H)||_{\omega_\Gamma^k\setminus\omega_\Gamma^{k-1}} \\
    &\qquad\qquad + H^{-1}||u- v_h||_{\omega_\Gamma^k\setminus\omega_\Gamma^{k-1}}+||\nabla(u-v_h)||_{\omega_\Gamma^k\setminus\omega_\Gamma^{k-1}}. \\
  \end{aligned}
\end{equation}
Furthermore
\begin{equation}
  \begin{aligned}
    &\kappa||\eta(v_h- v_H)||^2_{\partial(\omega_\Gamma^k\setminus\omega_\Gamma^{k-1})\cap\partial\Omega} \lesssim \kappa||v_h - v_H||^2_{\partial(\omega_\Gamma^k\setminus\omega_\Gamma^{k-1})\cap\partial\Omega} \\
    &\qquad  \lesssim \kappa H^{-1}||v_h - v_H||^2_{\omega_\Gamma^k\setminus\omega_\Gamma^{k-1}}+H||\nabla(v_h - v_H)||^2_{\omega_\Gamma^k\setminus\omega_\Gamma^{k-1}}.
  \end{aligned}
\end{equation}
Taking the infimum and using that
\begin{equation}
  \begin{aligned}
    &\inf_{v_h\in\V_h}\left(H^{-1}||u- v_h||_{\omega_\Gamma^k\setminus\omega_\Gamma^{k-1}}+||\nabla(u-v_h)||_{\omega_\Gamma^k\setminus\omega_\Gamma^{k-1}}\right) \\
    &\qquad \leq \inf_{v_H\in\V_H} \left(H^{-1}||u-v_H||_{\omega_\Gamma^k\setminus\omega_\Gamma^{k-1}}+||\nabla(u-v_H)||_{\omega_\Gamma^k\setminus\omega_\Gamma^{k-1}}\right),
  \end{aligned}
\end{equation}
since $h < H$ and $\V_H\subset V_h$, concludes the proof.
\end{proof}

The analysis extends to a non-polygonal boundary if we assume that
$h$ is fine enough to approximate the boundary using interpolation
onto piecewise affine functions, see e.g. \cite{Sc75}.

\subsection{Estimation of the error in the global multiscale method}
In this section we present and analyze the method with
non-localized correctors.
\begin{lemma}\label{lemma:stepontheway}
  Let $\uref\in\VGammah$ solve \eqref{eq:omegagamma} and $\uH\in
\V_H^\Gamma$ solve \eqref{eq:idealgamma}, then
  \begin{equation}
    |||\uref-\uH||| \lesssim H\|f\|_{\omega_\Gamma^k}
  \end{equation}
  holds.
\end{lemma}
\begin{proof}
  Any $\uref\in\VGammah$ can be uniquely written as
  $\uref=\uH + u^\mathrm{f}$ where $\uH\in
  \V^\Gamma_H$ and $u^\mathrm{f}\in \V^\mathrm{f}$. This follows from
  the result from functional analysis, that if we have a
  projection $\mathcal{P}:\uref\to \uH$ onto a closed
  subspace, there is a unique split $\uref=\mathcal{P}\uref +
  (1-\mathcal{P})\uref$. For $\mathcal{P}=(1+Q)\mathcal{I}_H$, we
  have $\mathcal{P}\VGammah=\V^\Gamma_H$ and
  \begin{equation}
    \mathcal{P}^2=(1+Q)\mathcal{I}_H(1+Q)\mathcal{I}_H=(1+Q)\mathcal{I}_H\mathcal{I}_H=(1+Q)\mathcal{I}_H=\mathcal{P}.
  \end{equation}
  We obtain
  \begin{equation}
    |||u^\mathrm{f}|||^2 = a(u^\mathrm{f},u^\mathrm{f}) = (f,u^\mathrm{f})_{L^2(\Omega)} = (f,u^\mathrm{f}-\mathcal{I}_Hu^\mathrm{f})_{L^2(\omega_\Gamma^k)} \lesssim H\|f\|_{\omega_\Gamma^k}|||u^\mathrm{f}|||,
  \end{equation}
  which concludes the proof.
\end{proof}

\subsection{Estimation of the error between global and localized correction}
The correctors fulfill the following decay property.
\begin{lemma}[Decay of correctors]\label{lem:decay}
  For any $x\in\N_I$ there exist a $0<\gamma<1$ such that the local
  corrector $Q^L_x(\varphi_x)\in \V^\mathrm{f}_L(\varphi_x)$ and the
  global corrector $Q(\varphi_x)\in\V^\mathrm{f}(\varphi_x)$, which
  solves \eqref{eq:local_corrector} and \eqref{eq:global_corrector}
  respectively, fulfills the decay property
  \begin{equation}
    |||(Q - Q^L)(v_H)|||^2 \leq L^{d}\gamma^{\left\lfloor (L-3)/3 \right\rfloor}\sum_{x\in\mathcal{N}_I} v_x^2|||Q \phi_x|||^2,
  \end{equation}
  where $\lfloor\cdot\rfloor$ is the floor function which maps a real number to the largest smaller integer.
\end{lemma}
\begin{proof}
  See Appendix~\ref{app:proofs}
\end{proof}
The localized corrected basis functions fulfill the following
stability property.
\begin{lemma}\label{lem:g}
  Under Assumption~\ref{ass:interior_element} we have the stability
  \begin{equation}
    ||| \varphi_x + Q^L(\varphi_x)||| \lesssim H^{-1}|| \varphi_x ||,
  \end{equation}
  for the corrected basis function given any $x\in\N_I$.
\end{lemma}
\begin{proof}
  First we will prove that there exist a (non-unique) function
  $g_x\in \V^f(\omega^L_x)$ such that
  $(g_x - \varphi_x)|_{\Gamma_D}=0$ and
  $|||g_x||| \lesssim |||\varphi_x|||$ for all $x\in\N_I$. Given any
  node $x$ define $w|_{T}=g_x-\varphi_x$ and $w|_{\Omega\setminus T}=0$
  where $T$ is an interior element. The function $w$ have to fulfill
  the following restriction
  \begin{equation}
    \mathcal{I}_Hw = \mathcal{I}_H\varphi_x = \varphi_x,
  \end{equation}
  which is equivalent to
  \begin{equation}
    (P_yw)(y) = \delta_{xy},
  \end{equation}
  where $\delta_{xy}$ is the Kronecker delta function.  In order to construct
  $w$ we perform two a uniform refinements in 2D. A similar construction is
  possible in 3D using two red-green refinements but we restrict the
  discussion to 2D for simplicity. Then we have three free nodes in $T$ for a
  function that is zero on the boundary $\partial T$. We write $w$ as
  $w=\sum_{j=1}^{d+1} \alpha_j\hat\varphi_j$ where $\varphi_j$ are the $P_1$
  Lagrange basis function associated with the three interior nodes. We can
  determine $w$ by letting it fulfill
  \begin{equation}
    \sum_{i,j=1}^{d+1} \alpha_j(P_{y_i}\hat\varphi_j)(y_i) = \delta_{x,y_i}.
  \end{equation}
  The value $(P_y\hat\varphi_j)(x)$ can be computed as
  \begin{equation}\label{eq:lagrange2}
    (P_y\hat\varphi_j)(y) = \delta_y^T(\Pi^T M_{H/4} \Pi)^{-1}\Pi^TM_{H/4}\hat\varphi_j,
  \end{equation}
  where $M_{H/4}$ is a local mass matrix computed on $\omega_x^0$,
  $\delta_x^T= 1$ for index $x$ and $0$ otherwise, and
  $\Pi:\V_{H}\to\V_{H/4}$ is a projection from the finer space onto
  the coarse. On a quasi-uniform mesh $P_y\hat\varphi_j(y)$ is
  independent of $H$. Therefore, $\alpha_j$ is independent of $H$
  and there exist a constant such that
  \begin{equation}
    ||w|| \leq C || \varphi_x ||.
  \end{equation}
  This yields
  \begin{equation}
    |||w||| \lesssim 4H^{-1}||w|| \lesssim H^{-1}|| \varphi_x ||.
  \end{equation}
  Using the triangle inequality we have
  \begin{equation}
    |||g_x||| \leq |||\varphi_x||| + |||w||| \lesssim H^{-1} ||\varphi_x||.
  \end{equation}
  Next we consider the problem: find $Q_0\in \V^f(\omega^L_x)$ such
  that
  \begin{equation}
    a(Q_0,w) = a(\varphi_x - g_x,w)\quad w\in\V^f(\omega^L_x),
  \end{equation}
  where $g_x$ satisfies $g_x\in\V^f(\omega^L_x)$,
  $(\varphi_x-g_x)|_{\Gamma_D}=0$, and $|||g_x|||\lesssim
  |||\varphi_x|||$. It is clear that $Q^L(\varphi_x)= Q_0+g$.
  For the stability we obtain
  \begin{equation}
    \begin{aligned}
      ||| \varphi_x + Q^L(\varphi_x)|||^2 &\leq a(\varphi_x + Q^L(\varphi_x),\varphi_x + Q^L(\varphi_x)) = a(\varphi_x + Q^L(\varphi_x),\varphi_x + Q_0 + g) \\
      & =  a(\varphi_x + Q^L(\varphi_x),\varphi_x + g) \leq  |||\varphi_x + Q^L(\varphi_x)|||(|||\varphi_x||| + |||g_x|||) \\
      & \lesssim |||\varphi_x + Q^L(\varphi_x)|||\cdot|||\varphi_x|||,
    \end{aligned}
  \end{equation}
  which concludes the proof.
\end{proof}
\subsection{Estimation of the error for the localized multiscale method}
The a priori results for the localized multiscale method reads.
\begin{theorem}\label{thm:local_enriched_fem}
  Let $u\in\V$ solve \eqref{eq:week} and
  $\uHL\in\V_H^{\Gamma,L}$ solve \eqref{eq:idealgamma}.
  Then under Assumption~\ref{ass:interior_element} the bound
  \begin{equation}
    \begin{aligned}
      |||u- \uHL||| &\lesssim \inf_{v\in \V_H}\left(\|H^{-1}(u - v)\|_{\Omega\setminus\omega_{\Gamma}^{k-1}}+|||u - v|||_{\Omega\setminus \omega_{\Gamma}^{k-1}}\right)\\
      &\qquad  + \inf_{v\in \V_h(\omega_{\Gamma}^k)}|||u-v|||_{\omega_{\Gamma}^{k}} + \|Hf\|_{\omega_\Gamma^k} + L^{d/2}H^{-1}\gamma^L\|f\|,
    \end{aligned}
  \end{equation}
  holds for $k\geq 2$.
\end{theorem}
\begin{proof}
  Since $\V_H^{\Gamma,L}\subset\V$ we have the best approximation result
  \begin{equation}
     |||u- \uHL||| \leq |||u- v_H|||\quad \text{for all }v_H\in\V_H^{\Gamma,L}.
  \end{equation}
  We obtain
  \begin{equation}\label{eq:main1}
    |||u- \uHL||| \leq |||u - \uref||| + |||\uref - \uH||| + |||\uH-v_H|||,
  \end{equation}
  using the triangle inequality. The first and second term is bounded
  using Lemma~\ref{lemma:H1enrichment} and
  \ref{lemma:stepontheway}. For the third term we choose
  $v_H=\uH + Q(\uH)$
  which gives
  \begin{equation}\label{eq:main2}
    |||\uH-v_H|||^2 = |||Q(\uH)-Q^L(\uH)|||^2\lesssim L^{d}\gamma^{2L}\sum_{x\in\mathcal{N}_I}v_x^2|||Q \phi_x|||^2,
  \end{equation}
  using Lemma~\ref{lem:decay}. Lemma~\ref{lem:g} now gives
  \begin{equation}\label{eq:main3}
    \begin{aligned}
    |||\uH-v_H|||^2 &\lesssim H^{-2}L^{d}\gamma^{2L}\sum_{x\in\mathcal{N}_I}\uH(x)^2\|\varphi_x\|^2 \lesssim L^{d}\gamma^{2L}\|\uH\|^2\\
    &     \lesssim H^{-2}L^{d}\gamma^{2L} |||\uH|||^2
     \leq H^{-2}L^{d}\gamma^{2L} ||f||^2,
    \end{aligned}
  \end{equation}
  where also a Poincar\'e-Friedrich inequality has been used. Combining \eqref{eq:main1}
  and \eqref{eq:main3} concludes the proof.
\end{proof}

\section{Implementation and conditioning}
\label{implement}
In this section we will shortly discuss how to implement the proposed method and
analyze the conditioning of the matrices. For a more detailed discussion on
implementation of LOD see \cite{EHMP14}.

\subsection{Implementation}\label{sec:implement}
To compute $Q^L(\varphi_x)$ in \eqref{eq:local_corrector} we impose
the extra condition $\mathcal{I}_Hv=0$ using Lagrangian
multipliers. Let $n_x$ and $N_x$ be the number of fine and coarse
degrees of freedom in the patch $\omega^0_x$. Let $M_{x}$ and
$K_{x}$ denote the local mass and stiffness matrix on $\omega^0_x$
satisfying
\begin{equation}
  (v,w)_{\omega^0_x}\quad \Leftrightarrow\quad \hat w^T M_{x}\hat v,
\end{equation}
and
\begin{equation}
  a(v,w)|_{\omega^0_x}\quad \Leftrightarrow\quad \hat w^TK_{x}\hat v,
\end{equation}
where $v,w\in\V_h|_{\omega^0_x}$ and
$\hat w,\hat v\in\mathbb{R}^{n_x}$ are the nodal values of $v,w$. We
also define the projection matrix
$\Pi_x:\{v\in\V_H(\omega^0_x)\mid
\text{supp}(v)\cap\omega^0_x\neq 0\}\to\V_h(\omega^0_x)$
of size $(n_x\times N_x)$ which project a coarse function onto the
fine mesh and a Kronecker delta vector of size $N_x\times 1$ as
$\delta_x=(0,\dots,0,1,0,\dots,0)$ where the $1$ is in node $x$. We obtain
\begin{equation}
P_xv(x)=0 \quad \Leftrightarrow\quad \lambda_x^T\hat v=\delta_x^T(\Pi_x^T\widehat M_{x}\Pi_x)^{-1}\Pi^T_xM_{x}\hat v = 0.
\end{equation}
Set
$\Lambda=[\lambda_{y_1}, \lambda_{y_1}, \dots, \lambda_{y_{N_x}}]$,
then \eqref{eq:local_corrector} is equivalent to solving the linear
system
\begin{equation}\label{eq:fine_system}
  \begin{bmatrix}
    K_x & \Lambda \\
    \Lambda^T & 0
  \end{bmatrix}
  \begin{bmatrix}
     \widehat Q^L(\varphi_x) \\
      \mu
  \end{bmatrix}
  =
  \begin{bmatrix}
     - K_x\Pi_x\widehat\varphi_x  \\
     0
  \end{bmatrix},
\end{equation}
where $\widehat Q^L(\varphi_x)\in\mathbb{R}^{n_x}$ contains the nodal values
of $Q^L(\varphi_x)$ and $\mu$ is a Lagrange multiplier. For each coarse node
$x$ in $\omega^0_x$ we need to invert $\Pi_x^T\widehat M_{x}\Pi_x$ to assemble
\eqref{eq:fine_system}, however since the size is only $(N_x\times N_x)$ it is
cheap. The coarse scale stiffness matrix $\widehat K$ in
\eqref{eq:local_multiscale} is given element wise by
\begin{equation}
  \widehat K_{i,j} = a(\varphi_j+Q^L(\varphi_j),\varphi_i+Q^L(\varphi_i)).
\end{equation}
We save the nodal values of the corrected basis 
\begin{equation}\label{eq:corr_basis}
  \Phi =
  \begin{bmatrix}
    \varphi_{y_1}+\widehat Q^L(\varphi_{y_1}),\varphi_{y_2}+\widehat Q^L(\varphi_{y_2})\dots,\varphi_{y_N}+\widehat Q^L(\varphi_{y_N})
  \end{bmatrix}.
\end{equation}
Given the fine scale stiffness matrix $K$ and the collection of
corrected basis functions $\Phi$, we can compute the coarse stiffness
matrix as
\begin{equation}
  \widehat K = \Phi^T K \Phi,
\end{equation}
and the linear system \eqref{eq:local_multiscale} is computed as
\begin{equation}
  \widehat K \hat u^{\Gamma,L}_{H} = b,
\end{equation}
where $\hat u^{\Gamma,L}_{H}$ is the nodal values of
$\hat u^{\Gamma,L}_{H}$ and $b$ correspond to the right hand side
$b_{y_1}=(f,\varphi_{y_1}+\widehat Q^L(\varphi_{y_1}))$.  However,
the fine stiffness matrix does not have to be assembled globally.
If a Petrov-Galerkin formulation is used, further savings can be
made \cite{ElGiHe15}.

\subsection{Conditioning}
The Euclidean matrix norm is defined as
\begin{equation}
  || A ||_N = \sup_{0\neq v\in\mathbb{R}^N} \frac{|Av|_N}{|v|_N},
\end{equation}
where $\langle v, w \rangle=\sum_{i=1}^N v(x_i)w(x_i)$ and
$|v|_N= \sqrt{\langle v, v \rangle}$.
\begin{theorem}
  The bound
  \begin{equation}
    \kappa = \|\widehat K\|_{N}\|\widehat K^{-1}\|_{N} \lesssim H^{-2},
  \end{equation}
  on the condition number $\kappa$ holds.
\end{theorem}
\begin{proof}
  To prove the bound of the condition number we use the following three properties.
  \begin{enumerate}
  \item An inverse type inequality for the modified basis functions. We
  have
    \begin{equation}
      \begin{aligned}
        ||| \varphi_i + Q(\varphi_i) ||| \lesssim H^{-1}\| \varphi_i\|,
      \end{aligned}
    \end{equation}
    from Lemma~\ref{lem:g}.
  \item A Poincar\'e-Friedrich type inequality on the full domain, see
    Section~\ref{sec:PF}.
  \item An equivalence between the Euclidean norm and the
    $L^2$-norm. We have that
    \begin{equation}
      \begin{aligned}
        \|v\|^2 & \leq \sum v_i^2\|\varphi_i + Q(\varphi_i)\|^2 \lesssim  \sum v_i^2(\|\varphi_i\| + \|Q(\varphi_i)-\mathcal{I}_HQ(\varphi_i)\|)^2 \\
        &  \lesssim \sum v_i^2(\|\varphi_i\| + H\|\nabla Q(\varphi_i)\|)^2 \\
        & \lesssim \sum v_i^2(\|\varphi_i\| + H\|\nabla \varphi_i\| + H\|\nabla(\varphi_i - Q(\varphi_i))\|)^2 \\
        & \lesssim \sum v_i^2\|\varphi_i\|^2 \lesssim H^d|v|_N,
      \end{aligned}
    \end{equation}
    and
    \begin{equation}
      |v|_N^2 = \sum_{i=1}^N v_i^2 \lesssim H^{-d}\sum_{i=1}^N v_i^2\|\varphi_i\|^2 \lesssim H^{-d} \|\sum_{i=1}^N v_i\varphi_i\|^2= H^{-d}\|\mathcal{I}_H v\|^2 \lesssim H^{-d}\| v\|^2.
    \end{equation}
    holds, hence $|v|_N \sim H^{-d/2}\|v\|$.
  \end{enumerate}
  We have
  \begin{equation}
    |\widehat Kv|_N = \sup_{0\neq w\in\mathbb{R}^N} \frac{|\langle\widehat Kv, w  \rangle|}{|w|_N} = \sup_{0\neq w\in\mathbb{R}^N} \frac{|a(v,w)|}{|w|_N} =  \sup_{0\neq w\in\mathbb{R}^N} \frac{|||v|||\cdot|||w|||}{|w|_N} \leq H^{d-2} |v|_N,
  \end{equation}
  using property 1) and 3). Also
  \begin{equation}
    \begin{aligned}
    |\widehat K^{-1}v|_N^2 &= H^{-d}\| \widehat K^{-1}v \|^2 \leq \CP H^{-d} ||| \widehat K^{-1}v |||^2 \\
    &\leq H^{-d}\langle \widehat K\widehat K^{-1}v ,\widehat K^{-1}v\rangle \leq H^{-d}|v|_N\cdot|\widehat K^{-1}v|_N,
    \end{aligned}
  \end{equation}
  using property 2) and 3). The proof is concluded by taking the supremum over $v$.
\end{proof}

\section{Numerical experiments}
\label{numerics}
In the following section we present some numerical experiments that verifies
our analytical results. In all the experiments we fix the right hand
side to $f=1$ and the localization parameter $L=\lceil1.5\log_2(H^{-1})
\rceil$.

\subsection{Accuracy on fractal shaped domain}
\label{sec:fractalnum}
We consider the domain in Figure~\ref{fig:fractal}. We use homogeneous
Dirichlet boundary conditions on the most left, down, and right hand side
boundaries and Robin boundary condition with $\kappa = 10$ on the rest, see
Figure~\ref{fig:fractal1}.
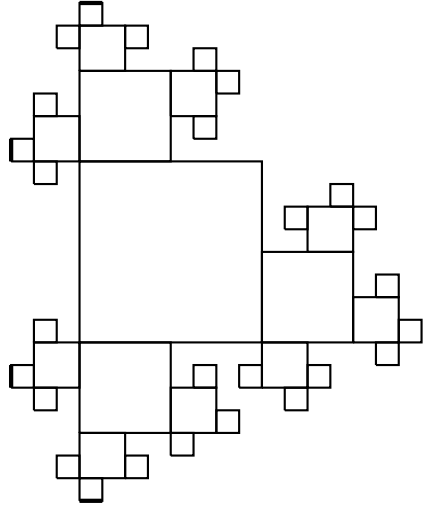
\begin{figure}[htb]
  \centering
\begin{tikzpicture}[scale=0.8]
  \coordinate [label={right:$ $}] (S1) at (0,1.5);
  \coordinate [label={right:$ $}] (S1) at (2.9,0.75);
  \coordinate [label={right:$ $}] (S1) at (3.7,0.375);
  \coordinate [label={right:$ $}] (S1) at (4.6,0.2);

  \newcommand{\sx}{0cm}
  \newcommand{\sy}{0cm}  
  \newcommand{\squarewidth}{3cm}
  \coordinate [label={below right:$ $}] (S1) at (\sx, \sy);
  \coordinate [label={above right:$ $}] (S2) at (\sx, \sy+\squarewidth);

  \coordinate [label={below right:$ $}] (S1) at (\sx, \sy);
  \coordinate [label={above right:$ $}] (S2) at (\sx, \sy+\squarewidth);
  \coordinate [label={below left:$  $}]  (S3) at (\sx+\squarewidth, \sy+\squarewidth);
  \coordinate [label={below left:$ $}]  (S4) at (\sx+\squarewidth, \sy);
  \draw [thick](S1) -- (S2) -- (S3) -- (S4) -- (S1);

  \renewcommand{\sx}{3cm}
  \renewcommand{\sy}{0cm}  
  \renewcommand{\squarewidth}{1.5cm}
  \coordinate [label={below right:$ $}] (S1) at (\sx, \sy);
  \coordinate [label={above right:$ $}] (S2) at (\sx, \sy+\squarewidth);
  \coordinate [label={below left:$ $}]  (S3) at (\sx+\squarewidth, \sy+\squarewidth);
  \coordinate [label={below left:$ $}]  (S4) at (\sx+\squarewidth, \sy);
  \draw [thick] (S1) -- (S2) -- (S3) -- (S4) -- (S1);

  \renewcommand{\sx}{0cm}
  \renewcommand{\sy}{3cm}  
  \renewcommand{\squarewidth}{1.5cm}
  \coordinate [label={below right:$ $}] (S1) at (\sx, \sy);
  \coordinate [label={above right:$ $}] (S2) at (\sx, \sy+\squarewidth);
  \coordinate [label={below left:$ $}]  (S3) at (\sx+\squarewidth, \sy+\squarewidth);
  \coordinate [label={below left:$ $}]  (S4) at (\sx+\squarewidth, \sy);
  \draw [thick] (S1) -- (S2) -- (S3) -- (S4) -- (S1);

  \renewcommand{\sx}{0cm}
  \renewcommand{\sy}{-1.5cm}  
  \renewcommand{\squarewidth}{1.5cm}
  \coordinate [label={below right:$ $}] (S1) at (\sx, \sy);
  \coordinate [label={above right:$ $}] (S2) at (\sx, \sy+\squarewidth);
  \coordinate [label={below left:$ $}]  (S3) at (\sx+\squarewidth, \sy+\squarewidth);
  \coordinate [label={below left:$ $}]  (S4) at (\sx+\squarewidth, \sy);
  \draw [thick] (S1) -- (S2) -- (S3) -- (S4) -- (S1);

  \renewcommand{\sx}{4.5cm}
  \renewcommand{\sy}{0cm}  
  \renewcommand{\squarewidth}{0.75cm}
  \coordinate [label={below right:$ $}] (S1) at (\sx, \sy);
  \coordinate [label={above right:$ $}] (S2) at (\sx, \sy+\squarewidth);
  \coordinate [label={below left:\scriptsize$ $}]  (S3) at (\sx+\squarewidth, \sy+\squarewidth);
  \coordinate [label={below left:$ $}]  (S4) at (\sx+\squarewidth, \sy);
  \draw [thick] (S1) -- (S2) -- (S3) -- (S4) -- (S1);

  \renewcommand{\sx}{0cm}
  \renewcommand{\sy}{-2.25cm}  
  \renewcommand{\squarewidth}{0.75cm}
  \coordinate [label={below right:$ $}] (S1) at (\sx, \sy);
  \coordinate [label={above right:$ $}] (S2) at (\sx, \sy+\squarewidth);
  \coordinate [label={below left:$ $}]  (S3) at (\sx+\squarewidth, \sy+\squarewidth);
  \coordinate [label={below left:$ $}]  (S4) at (\sx+\squarewidth, \sy);
  \draw [thick] (S1) -- (S2) -- (S3) -- (S4) -- (S1);

  \renewcommand{\sx}{0-0.75cm}
  \renewcommand{\sy}{-0.75cm}  
  \renewcommand{\squarewidth}{0.75cm}
  \coordinate [label={below right:$ $}] (S1) at (\sx, \sy);
  \coordinate [label={above right:$ $}] (S2) at (\sx, \sy+\squarewidth);
  \coordinate [label={below left:$ $}]  (S3) at (\sx+\squarewidth, \sy+\squarewidth);
  \coordinate [label={below left:$ $}]  (S4) at (\sx+\squarewidth, \sy);
  \draw [thick] (S1) -- (S2) -- (S3) -- (S4) -- (S1);

  \renewcommand{\sx}{1.5cm}
  \renewcommand{\sy}{-1.5cm}  
  \renewcommand{\squarewidth}{0.75cm}
  \coordinate [label={below right:$ $}] (S1) at (\sx, \sy);
  \coordinate [label={above right:$ $}] (S2) at (\sx, \sy+\squarewidth);
  \coordinate [label={below left:$ $}]  (S3) at (\sx+\squarewidth, \sy+\squarewidth);
  \coordinate [label={below left:$ $}]  (S4) at (\sx+\squarewidth, \sy);
  \draw [thick] (S1) -- (S2) -- (S3) -- (S4) -- (S1);

  \renewcommand{\sx}{3cm}
  \renewcommand{\sy}{-0.75cm}  
  \renewcommand{\squarewidth}{0.75cm}
  \coordinate [label={below right:$ $}] (S1) at (\sx, \sy);
  \coordinate [label={above right:$ $}] (S2) at (\sx, \sy+\squarewidth);
  \coordinate [label={below left:$ $}]  (S3) at (\sx+\squarewidth, \sy+\squarewidth);
  \coordinate [label={below left:$ $}]  (S4) at (\sx+\squarewidth, \sy);
  \draw [thick] (S1) -- (S2) -- (S3) -- (S4) -- (S1);

  \renewcommand{\sx}{3.75cm}
  \renewcommand{\sy}{1.5cm}  
  \renewcommand{\squarewidth}{0.75cm}
  \coordinate [label={below right:$ $}] (S1) at (\sx, \sy);
  \coordinate [label={above right:$ $}] (S2) at (\sx, \sy+\squarewidth);
  \coordinate [label={below left:$ $}]  (S3) at (\sx+\squarewidth, \sy+\squarewidth);
  \coordinate [label={below left:$ $}]  (S4) at (\sx+\squarewidth, \sy);
  \draw [thick] (S1) -- (S2) -- (S3) -- (S4) -- (S1);

  \renewcommand{\sx}{0cm}
  \renewcommand{\sy}{4.5cm}  
  \renewcommand{\squarewidth}{0.75cm}
  \coordinate [label={below right:$ $}] (S1) at (\sx, \sy);
  \coordinate [label={above right:$ $}] (S2) at (\sx, \sy+\squarewidth);
  \coordinate [label={below left:$ $}]  (S3) at (\sx+\squarewidth, \sy+\squarewidth);
  \coordinate [label={below left:$ $}]  (S4) at (\sx+\squarewidth, \sy);
  \draw [thick] (S1) -- (S2) -- (S3) -- (S4) -- (S1);

  \renewcommand{\sx}{1.5cm}
  \renewcommand{\sy}{3.75cm}  
  \renewcommand{\squarewidth}{0.75cm}
  \coordinate [label={below right:$ $}] (S1) at (\sx, \sy);
  \coordinate [label={above right:$ $}] (S2) at (\sx, \sy+\squarewidth);
  \coordinate [label={below left:$ $}]  (S3) at (\sx+\squarewidth, \sy+\squarewidth);
  \coordinate [label={below left:$ $}]  (S4) at (\sx+\squarewidth, \sy);
  \draw [thick] (S1) -- (S2) -- (S3) -- (S4) -- (S1);

  \renewcommand{\sx}{-.75cm}
  \renewcommand{\sy}{3cm}  
  \renewcommand{\squarewidth}{0.75cm}
  \coordinate [label={below right:$ $}] (S1) at (\sx, \sy);
  \coordinate [label={above right:$ $}] (S2) at (\sx, \sy+\squarewidth);
  \coordinate [label={below left:$ $}]  (S3) at (\sx+\squarewidth, \sy+\squarewidth);
  \coordinate [label={below left:$ $}]  (S4) at (\sx+\squarewidth, \sy);
  \draw [thick] (S1) -- (S2) -- (S3) -- (S4) -- (S1);

  \renewcommand{\sx}{5.25cm}
  \renewcommand{\sy}{0cm}  
  \renewcommand{\squarewidth}{0.375cm}
  \coordinate [label={below right:$ $}] (S1) at (\sx, \sy);
  \coordinate [label={above right:$ $}] (S2) at (\sx, \sy+\squarewidth);
  \coordinate [label={below left:\scriptsize$ $}]  (S3) at (\sx+\squarewidth, \sy+\squarewidth);
  \coordinate [label={below left:$ $}]  (S4) at (\sx+\squarewidth, \sy);
  \draw [thick] (S1) -- (S2) -- (S3) -- (S4) -- (S1);  

  \coordinate [label={right:$ $}] (S5) at (\sx+\squarewidth+.15cm, \sy);
  \coordinate [label={below right:$ $}] (S6) at (\sx+\squarewidth+.15cm, \sy+\squarewidth);

  \renewcommand{\sx}{4.875cm}
  \renewcommand{\sy}{0.75cm}  
  \renewcommand{\squarewidth}{0.375cm}
  \coordinate [label={below right:$ $}] (S1) at (\sx, \sy);
  \coordinate [label={above right:$ $}] (S2) at (\sx, \sy+\squarewidth);
  \coordinate [label={below left:\scriptsize$ $}]  (S3) at (\sx+\squarewidth, \sy+\squarewidth);
  \coordinate [label={below left:$ $}]  (S4) at (\sx+\squarewidth, \sy);
  \draw [thick] (S1) -- (S2) -- (S3) -- (S4) -- (S1);  

  \renewcommand{\sx}{4.875cm}
  \renewcommand{\sy}{-0.375cm}  
  \renewcommand{\squarewidth}{0.375cm}
  \coordinate [label={below right:$ $}] (S1) at (\sx, \sy);
  \coordinate [label={above right:$ $}] (S2) at (\sx, \sy+\squarewidth);
  \coordinate [label={below left:\scriptsize$ $}]  (S3) at (\sx+\squarewidth, \sy+\squarewidth);
  \coordinate [label={below left:$ $}]  (S4) at (\sx+\squarewidth, \sy);
  \draw [thick] (S1) -- (S2) -- (S3) -- (S4) -- (S1);  

  \renewcommand{\sx}{0cm}
  \renewcommand{\sy}{5.25cm}  
  \renewcommand{\squarewidth}{0.375cm}
  \coordinate [label={below right:$ $}] (S1) at (\sx, \sy);
  \coordinate [label={above right:$ $}] (S2) at (\sx, \sy+\squarewidth);
  \coordinate [label={below left:\scriptsize$ $}]  (S3) at (\sx+\squarewidth, \sy+\squarewidth);
  \coordinate [label={below left:$ $}]  (S4) at (\sx+\squarewidth, \sy);
  \draw [thick] (S1) -- (S2) -- (S3) -- (S4) -- (S1);  
  \draw [ultra thick] (S2) -- (S3);

  \renewcommand{\sx}{-0.375cm}
  \renewcommand{\sy}{4.8750cm}  
  \renewcommand{\squarewidth}{0.375cm}
  \coordinate [label={below right:$ $}] (S1) at (\sx, \sy);
  \coordinate [label={above right:$ $}] (S2) at (\sx, \sy+\squarewidth);
  \coordinate [label={below left:\scriptsize$ $}]  (S3) at (\sx+\squarewidth, \sy+\squarewidth);
  \coordinate [label={below left:$ $}]  (S4) at (\sx+\squarewidth, \sy);
  \draw [thick] (S1) -- (S2) -- (S3) -- (S4) -- (S1);  

  \renewcommand{\sx}{0.75cm}
  \renewcommand{\sy}{4.8750cm}  
  \renewcommand{\squarewidth}{0.375cm}
  \coordinate [label={below right:$ $}] (S1) at (\sx, \sy);
  \coordinate [label={above right:$ $}] (S2) at (\sx, \sy+\squarewidth);
  \coordinate [label={below left:\scriptsize$ $}]  (S3) at (\sx+\squarewidth, \sy+\squarewidth);
  \coordinate [label={below left:$ $}]  (S4) at (\sx+\squarewidth, \sy);
  \draw [thick] (S1) -- (S2) -- (S3) -- (S4) -- (S1);  

  \renewcommand{\sx}{2.25cm}
  \renewcommand{\sy}{4.1250cm}  
  \renewcommand{\squarewidth}{0.375cm}
  \coordinate [label={below right:$ $}] (S1) at (\sx, \sy);
  \coordinate [label={above right:$ $}] (S2) at (\sx, \sy+\squarewidth);
  \coordinate [label={below left:\scriptsize$ $}]  (S3) at (\sx+\squarewidth, \sy+\squarewidth);
  \coordinate [label={below left:$ $}]  (S4) at (\sx+\squarewidth, \sy);
  \draw [thick] (S1) -- (S2) -- (S3) -- (S4) -- (S1);  

  \renewcommand{\sx}{1.875cm}
  \renewcommand{\sy}{4.5cm}  
  \renewcommand{\squarewidth}{0.375cm}
  \coordinate [label={below right:$ $}] (S1) at (\sx, \sy);
  \coordinate [label={above right:$ $}] (S2) at (\sx, \sy+\squarewidth);
  \coordinate [label={below left:\scriptsize$ $}]  (S3) at (\sx+\squarewidth, \sy+\squarewidth);
  \coordinate [label={below left:$ $}]  (S4) at (\sx+\squarewidth, \sy);
  \draw [thick] (S1) -- (S2) -- (S3) -- (S4) -- (S1);  

  \renewcommand{\sx}{1.875cm}
  \renewcommand{\sy}{3.375cm}  
  \renewcommand{\squarewidth}{0.375cm}
  \coordinate [label={below right:$ $}] (S1) at (\sx, \sy);
  \coordinate [label={above right:$ $}] (S2) at (\sx, \sy+\squarewidth);
  \coordinate [label={below left:\scriptsize$ $}]  (S3) at (\sx+\squarewidth, \sy+\squarewidth);
  \coordinate [label={below left:$ $}]  (S4) at (\sx+\squarewidth, \sy);
  \draw [thick] (S1) -- (S2) -- (S3) -- (S4) -- (S1);  

  \renewcommand{\sx}{-1.125cm}
  \renewcommand{\sy}{3cm}  
  \renewcommand{\squarewidth}{0.375cm}
  \coordinate [label={below right:$ $}] (S1) at (\sx, \sy);
  \coordinate [label={above right:$ $}] (S2) at (\sx, \sy+\squarewidth);
  \coordinate [label={below left:\scriptsize$ $}]  (S3) at (\sx+\squarewidth, \sy+\squarewidth);
  \coordinate [label={below left:$ $}]  (S4) at (\sx+\squarewidth, \sy);
  \draw [thick] (S1) -- (S2) -- (S3) -- (S4) -- (S1);  
  \draw [ultra thick] (S1) -- (S2);

  \renewcommand{\sx}{-0.75cm}
  \renewcommand{\sy}{3.75cm}  
  \renewcommand{\squarewidth}{0.375cm}
  \coordinate [label={below right:$ $}] (S1) at (\sx, \sy);
  \coordinate [label={above right:$ $}] (S2) at (\sx, \sy+\squarewidth);
  \coordinate [label={below left:\scriptsize$ $}]  (S3) at (\sx+\squarewidth, \sy+\squarewidth);
  \coordinate [label={below left:$ $}]  (S4) at (\sx+\squarewidth, \sy);
  \draw [thick] (S1) -- (S2) -- (S3) -- (S4) -- (S1);  

  \renewcommand{\sx}{-0.75cm}
  \renewcommand{\sy}{2.6247cm}  
  \renewcommand{\squarewidth}{0.375cm}
  \coordinate [label={below right:$ $}] (S1) at (\sx, \sy);
  \coordinate [label={above right:$ $}] (S2) at (\sx, \sy+\squarewidth);
  \coordinate [label={below left:\scriptsize$ $}]  (S3) at (\sx+\squarewidth, \sy+\squarewidth);
  \coordinate [label={below left:$ $}]  (S4) at (\sx+\squarewidth, \sy);
  \draw [thick] (S1) -- (S2) -- (S3) -- (S4) -- (S1);  

  \renewcommand{\sx}{-1.125cm}
  \renewcommand{\sy}{-.75cm}  
  \renewcommand{\squarewidth}{0.375cm}
  \coordinate [label={below right:$ $}] (S1) at (\sx, \sy);
  \coordinate [label={above right:$ $}] (S2) at (\sx, \sy+\squarewidth);
  \coordinate [label={below left:\scriptsize$ $}]  (S3) at (\sx+\squarewidth, \sy+\squarewidth);
  \coordinate [label={below left:$ $}]  (S4) at (\sx+\squarewidth, \sy);
  \draw [thick] (S1) -- (S2) -- (S3) -- (S4) -- (S1);  
  \draw [ultra thick] (S1) -- (S2);

  \renewcommand{\sx}{-0.75cm}
  \renewcommand{\sy}{0cm}  
  \renewcommand{\squarewidth}{0.375cm}
  \coordinate [label={below right:$ $}] (S1) at (\sx, \sy);
  \coordinate [label={above right:$ $}] (S2) at (\sx, \sy+\squarewidth);
  \coordinate [label={below left:\scriptsize$ $}]  (S3) at (\sx+\squarewidth, \sy+\squarewidth);
  \coordinate [label={below left:$ $}]  (S4) at (\sx+\squarewidth, \sy);
  \draw [thick] (S1) -- (S2) -- (S3) -- (S4) -- (S1);  

  \renewcommand{\sx}{-0.75cm}
  \renewcommand{\sy}{-1.1253cm}  
  \renewcommand{\squarewidth}{0.375cm}
  \coordinate [label={below right:$ $}] (S1) at (\sx, \sy);
  \coordinate [label={above right:$ $}] (S2) at (\sx, \sy+\squarewidth);
  \coordinate [label={below left:\scriptsize$ $}]  (S3) at (\sx+\squarewidth, \sy+\squarewidth);
  \coordinate [label={below left:$ $}]  (S4) at (\sx+\squarewidth, \sy);
  \draw [thick] (S1) -- (S2) -- (S3) -- (S4) -- (S1);  

  \renewcommand{\sx}{0cm}
  \renewcommand{\sy}{-2.6250cm}  
  \renewcommand{\squarewidth}{0.375cm}
  \coordinate [label={below right:$ $}] (S1) at (\sx, \sy);
  \coordinate [label={above right:$ $}] (S2) at (\sx, \sy+\squarewidth);
  \coordinate [label={below left:\scriptsize$ $}]  (S3) at (\sx+\squarewidth, \sy+\squarewidth);
  \coordinate [label={below left:$ $}]  (S4) at (\sx+\squarewidth, \sy);
  \draw [thick] (S1) -- (S2) -- (S3) -- (S4) -- (S1);  
  \draw [ultra thick] (S4) -- (S1);

  \renewcommand{\sx}{-0.375cm}
  \renewcommand{\sy}{-2.25cm}  
  \renewcommand{\squarewidth}{0.375cm}
  \coordinate [label={below right:$ $}] (S1) at (\sx, \sy);
  \coordinate [label={above right:$ $}] (S2) at (\sx, \sy+\squarewidth);
  \coordinate [label={below left:\scriptsize$ $}]  (S3) at (\sx+\squarewidth, \sy+\squarewidth);
  \coordinate [label={below left:$ $}]  (S4) at (\sx+\squarewidth, \sy);
  \draw [thick] (S1) -- (S2) -- (S3) -- (S4) -- (S1);  

  \renewcommand{\sx}{0.75cm}
  \renewcommand{\sy}{-2.25cm}  
  \renewcommand{\squarewidth}{0.375cm}
  \coordinate [label={below right:$ $}] (S1) at (\sx, \sy);
  \coordinate [label={above right:$ $}] (S2) at (\sx, \sy+\squarewidth);
  \coordinate [label={below left:\scriptsize$ $}]  (S3) at (\sx+\squarewidth, \sy+\squarewidth);
  \coordinate [label={below left:$ $}]  (S4) at (\sx+\squarewidth, \sy);
  \draw [thick] (S1) -- (S2) -- (S3) -- (S4) -- (S1);  

  \renewcommand{\sx}{1.5cm}
  \renewcommand{\sy}{-1.8750cm}  
  \renewcommand{\squarewidth}{0.375cm}
  \coordinate [label={below right:$ $}] (S1) at (\sx, \sy);
  \coordinate [label={above right:$ $}] (S2) at (\sx, \sy+\squarewidth);
  \coordinate [label={below left:\scriptsize$ $}]  (S3) at (\sx+\squarewidth, \sy+\squarewidth);
  \coordinate [label={below left:$ $}]  (S4) at (\sx+\squarewidth, \sy);
  \draw [thick] (S1) -- (S2) -- (S3) -- (S4) -- (S1);  

  \renewcommand{\sx}{2.25cm}
  \renewcommand{\sy}{-1.5cm}  
  \renewcommand{\squarewidth}{0.375cm}
  \coordinate [label={below right:$ $}] (S1) at (\sx, \sy);
  \coordinate [label={above right:$ $}] (S2) at (\sx, \sy+\squarewidth);
  \coordinate [label={below left:\scriptsize$ $}]  (S3) at (\sx+\squarewidth, \sy+\squarewidth);
  \coordinate [label={below left:$ $}]  (S4) at (\sx+\squarewidth, \sy);
  \draw [thick] (S1) -- (S2) -- (S3) -- (S4) -- (S1);  

  \renewcommand{\sx}{1.875cm}
  \renewcommand{\sy}{-0.75cm}  
  \renewcommand{\squarewidth}{0.375cm}
  \coordinate [label={below right:$ $}] (S1) at (\sx, \sy);
  \coordinate [label={above right:$ $}] (S2) at (\sx, \sy+\squarewidth);
  \coordinate [label={below left:\scriptsize$ $}]  (S3) at (\sx+\squarewidth, \sy+\squarewidth);
  \coordinate [label={below left:$ $}]  (S4) at (\sx+\squarewidth, \sy);
  \draw [thick] (S1) -- (S2) -- (S3) -- (S4) -- (S1);  

  \renewcommand{\sx}{2.6250cm}
  \renewcommand{\sy}{-0.75cm}  
  \renewcommand{\squarewidth}{0.375cm}
  \coordinate [label={below right:$ $}] (S1) at (\sx, \sy);
  \coordinate [label={above right:$ $}] (S2) at (\sx, \sy+\squarewidth);
  \coordinate [label={below left:\scriptsize$ $}]  (S3) at (\sx+\squarewidth, \sy+\squarewidth);
  \coordinate [label={below left:$ $}]  (S4) at (\sx+\squarewidth, \sy);
  \draw [thick] (S1) -- (S2) -- (S3) -- (S4) -- (S1);  

  \renewcommand{\sx}{3.75cm}
  \renewcommand{\sy}{-0.75cm}  
  \renewcommand{\squarewidth}{0.375cm}
  \coordinate [label={below right:$ $}] (S1) at (\sx, \sy);
  \coordinate [label={above right:$ $}] (S2) at (\sx, \sy+\squarewidth);
  \coordinate [label={below left:\scriptsize$ $}]  (S3) at (\sx+\squarewidth, \sy+\squarewidth);
  \coordinate [label={below left:$ $}]  (S4) at (\sx+\squarewidth, \sy);
  \draw [thick] (S1) -- (S2) -- (S3) -- (S4) -- (S1);  

  \renewcommand{\sx}{3.3750cm}
  \renewcommand{\sy}{-1.125cm}  
  \renewcommand{\squarewidth}{0.375cm}
  \coordinate [label={below right:$ $}] (S1) at (\sx, \sy);
  \coordinate [label={above right:$ $}] (S2) at (\sx, \sy+\squarewidth);
  \coordinate [label={below left:\scriptsize$ $}]  (S3) at (\sx+\squarewidth, \sy+\squarewidth);
  \coordinate [label={below left:$ $}]  (S4) at (\sx+\squarewidth, \sy);
  \draw [thick] (S1) -- (S2) -- (S3) -- (S4) -- (S1);  

  \renewcommand{\sx}{4.1250cm}
  \renewcommand{\sy}{2.25cm}  
  \renewcommand{\squarewidth}{0.375cm}
  \coordinate [label={below right:$ $}] (S1) at (\sx, \sy);
  \coordinate [label={above right:$ $}] (S2) at (\sx, \sy+\squarewidth);
  \coordinate [label={below left:\scriptsize$ $}]  (S3) at (\sx+\squarewidth, \sy+\squarewidth);
  \coordinate [label={below left:$ $}]  (S4) at (\sx+\squarewidth, \sy);
  \draw [thick] (S1) -- (S2) -- (S3) -- (S4) -- (S1);  

  \renewcommand{\sx}{4.5cm}
  \renewcommand{\sy}{1.875cm}  
  \renewcommand{\squarewidth}{0.375cm}
  \coordinate [label={below right:$ $}] (S1) at (\sx, \sy);
  \coordinate [label={above right:$ $}] (S2) at (\sx, \sy+\squarewidth);
  \coordinate [label={below left:\scriptsize$ $}]  (S3) at (\sx+\squarewidth, \sy+\squarewidth);
  \coordinate [label={below left:$ $}]  (S4) at (\sx+\squarewidth, \sy);
  \draw [thick] (S1) -- (S2) -- (S3) -- (S4) -- (S1);  

  \renewcommand{\sx}{3.3750cm}
  \renewcommand{\sy}{1.875cm}  
  \renewcommand{\squarewidth}{0.375cm}
  \coordinate [label={below right:$ $}] (S1) at (\sx, \sy);
  \coordinate [label={above right:$ $}] (S2) at (\sx, \sy+\squarewidth);
  \coordinate [label={below left:\scriptsize$ $}]  (S3) at (\sx+\squarewidth, \sy+\squarewidth);
  \coordinate [label={below left:$ $}]  (S4) at (\sx+\squarewidth, \sy);
  \draw [thick] (S1) -- (S2) -- (S3) -- (S4) -- (S1);  
  \end{tikzpicture}
  \caption{The fractal computational domain used in
  Section~\ref{sec:fractalnum}. We use homogeneous Dirichlet boundary conditions
  on the boundary segments marked by a thicker line and Robin boundary
  conditions elsewhere.}
  \label{fig:fractal1}
\end{figure}
We compute localized correctors in the full domain. The reference solution is
computed using $h=2^{-9}$. As seen in Figure~\ref{fig:fraktal}, a even higher
convergence rate than linear convergence to the reference solution is observed. We also see that the
correct scaling of the condition number with respect to the
coarse mesh parameter $H$.
\begin{figure}[htb]
  \centering
  \includegraphics{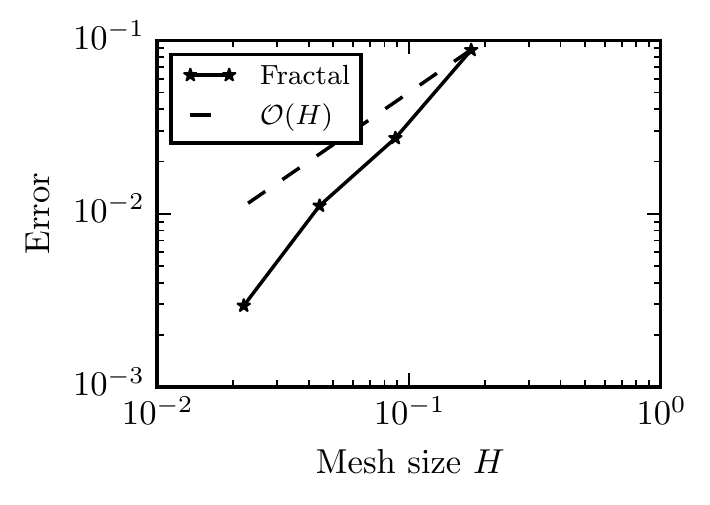}
  \includegraphics{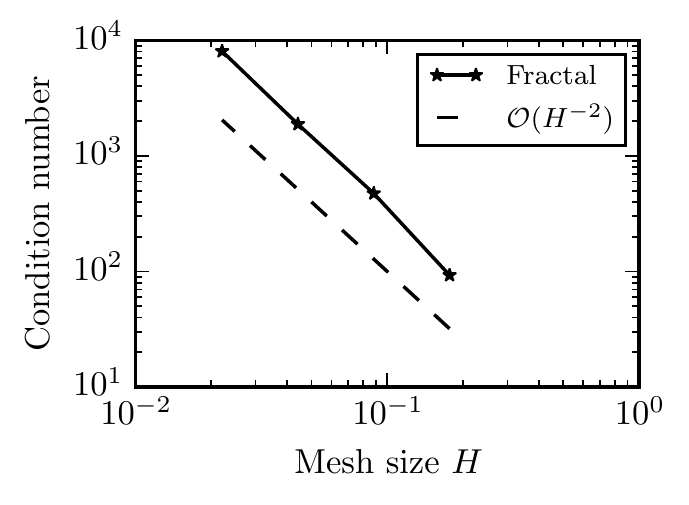}
  \caption{The convergence rate to the reference solution in relative
    energy norm (left) and the scaling of the condition number (right)
    for the fractal domain in Figure~\ref{fig:fraktal}.}
  \label{fig:fraktal}
\end{figure}

\subsection{Locally added correctors around singularities}
Let us consider two different domains, a domain with a re-entrant corner
$([0,1]\times[0,1])\backslash ([0.5,1]\times[0,0.5])$ and a slit-domain
$([0,1]\times[0,1])\backslash ([0.5,0.5]\times[0,0.5])$ with homogeneous
Dirichlet boundary condition. We only compute correctors in a vicinity of the
singularities, and denote this domain $\omega^k$. The size of $\omega^k$ can be determined by the size
of $|u|_{H^2(\omega^k)}$ (see Section~\ref{sec:reference}), i.e., how fast the
singularity decay away from the boundary. We choose $\omega^k$ as in
Figure~\ref{fig:singularitiesArea}.

As seen in Figure~\ref{fig:singularities}, the correct
convergence rates to the reference solutions and the correct scaling of the
condition number are observed for both singularities.
\begin{figure}[htb]
  \centering
  \includegraphics[scale = 0.3]{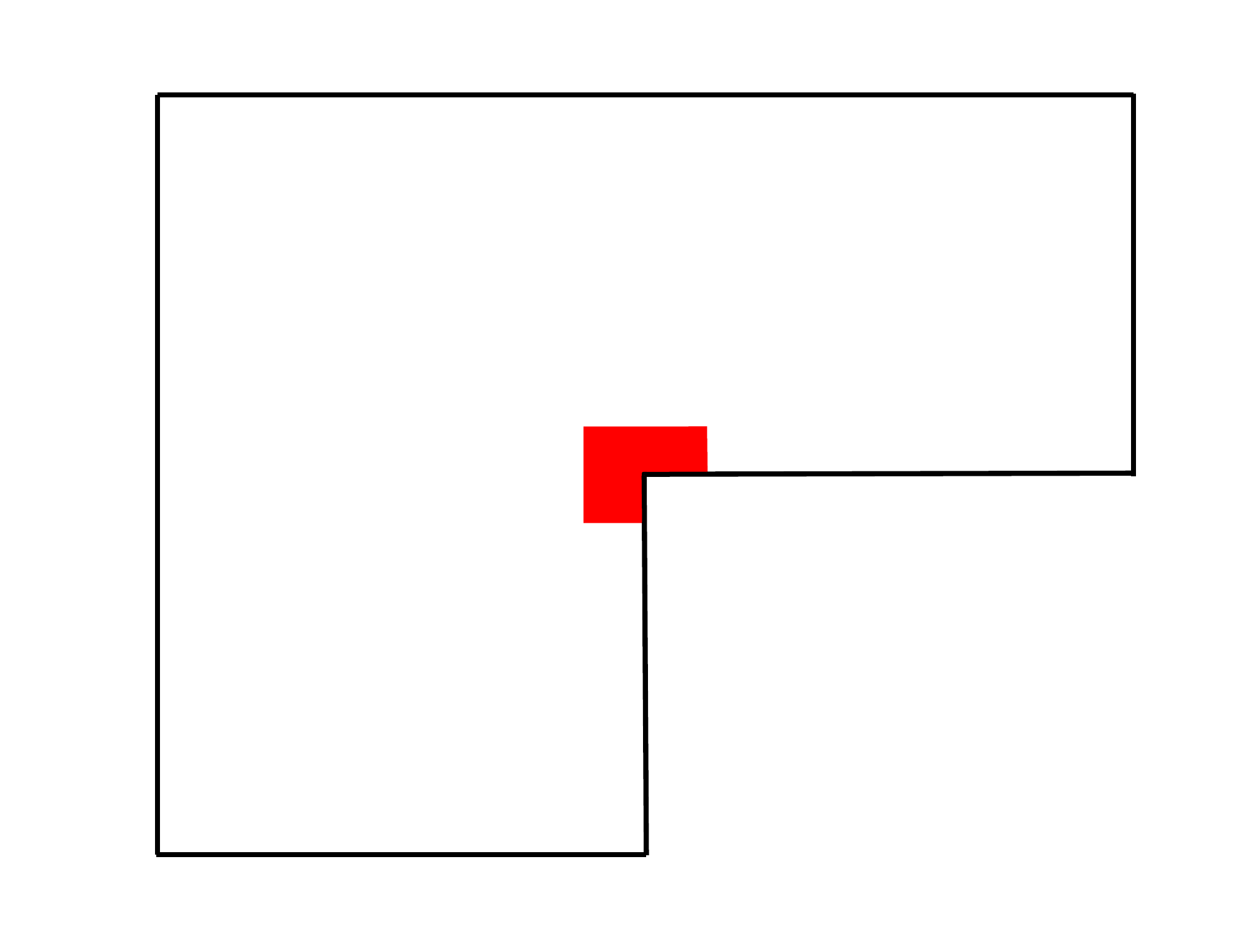}
  \includegraphics[scale = 0.3]{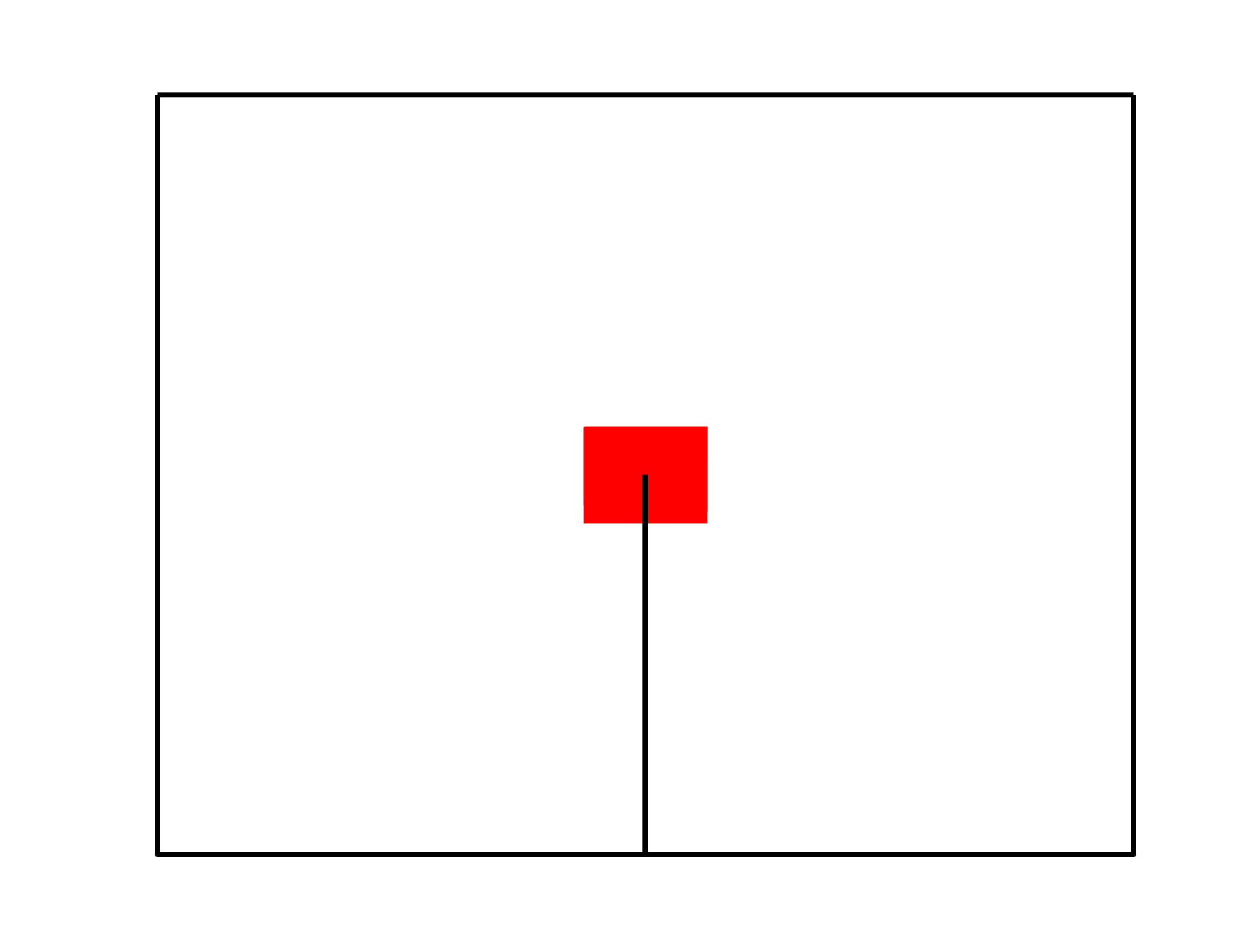}
  \caption{The marked area $\omega^k$ is where the finite element space is enriched,
    for domain with a re-entrant corner (left) and domain with a slit
    (right).}
  \label{fig:singularitiesArea}
\end{figure}
\begin{figure}[htb]
  \centering
  \includegraphics{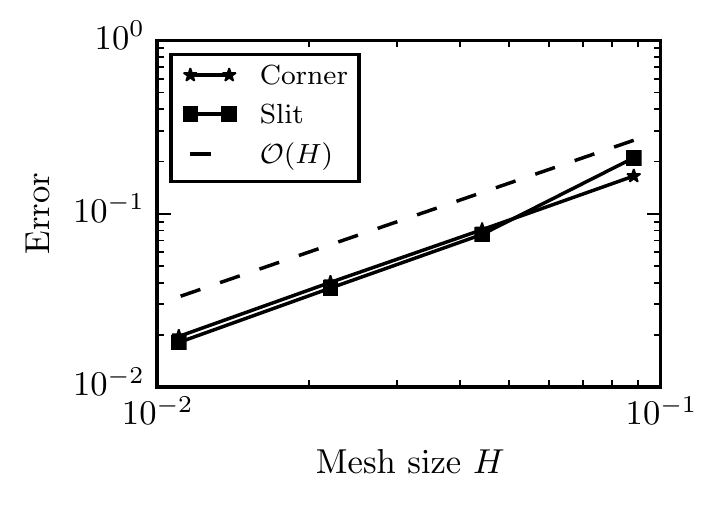}
  \includegraphics{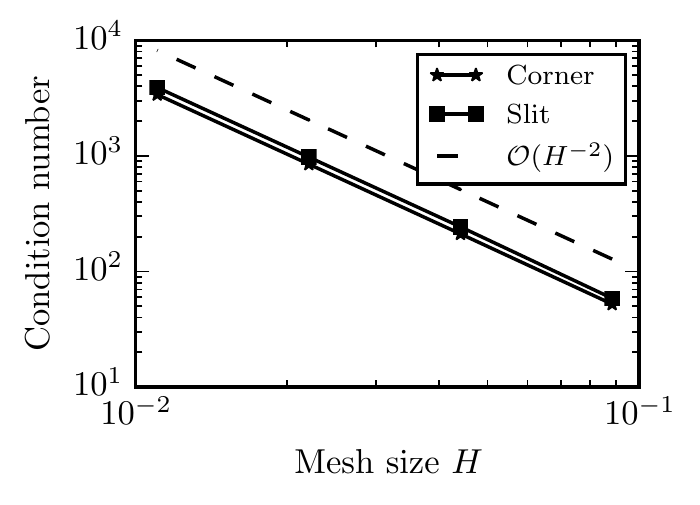}
  \caption{The convergence rate to the reference solution in relative
    energy norm (left) and scaling of the condition number (right) for
    the re-entrant corner and the slit domain.}
  \label{fig:singularities}
\end{figure}

\subsection{Locally add correctors around saw tooth boundary}
Let us consider a unit square where one of the boundaries are cut as a saw
tooth and where correctors are only computed in a vicinity of the saw teeth
shown in Figure~\ref{fig:sawDomainArea}. On all the non saw tooth boundaries
we use homogeneous Dirichlet boundary conditions. On the saw tooth boundary we
test both homogeneous Dirichlet and Neumann boundary condition. We observe the
correct convergence (although the Neumann case is a bit more sensitive and
might need slightly larger domain where correctors are computed) and scaling
of the condition number, see Figure~\ref{fig:sawconvergence}.
\begin{figure}[htb]
  \centering
  \includegraphics[height=4cm, width=4cm]{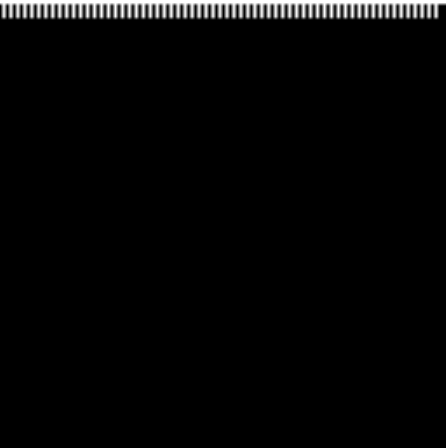}
  \hspace{2cm}
  \includegraphics[height=4cm, width=4cm]{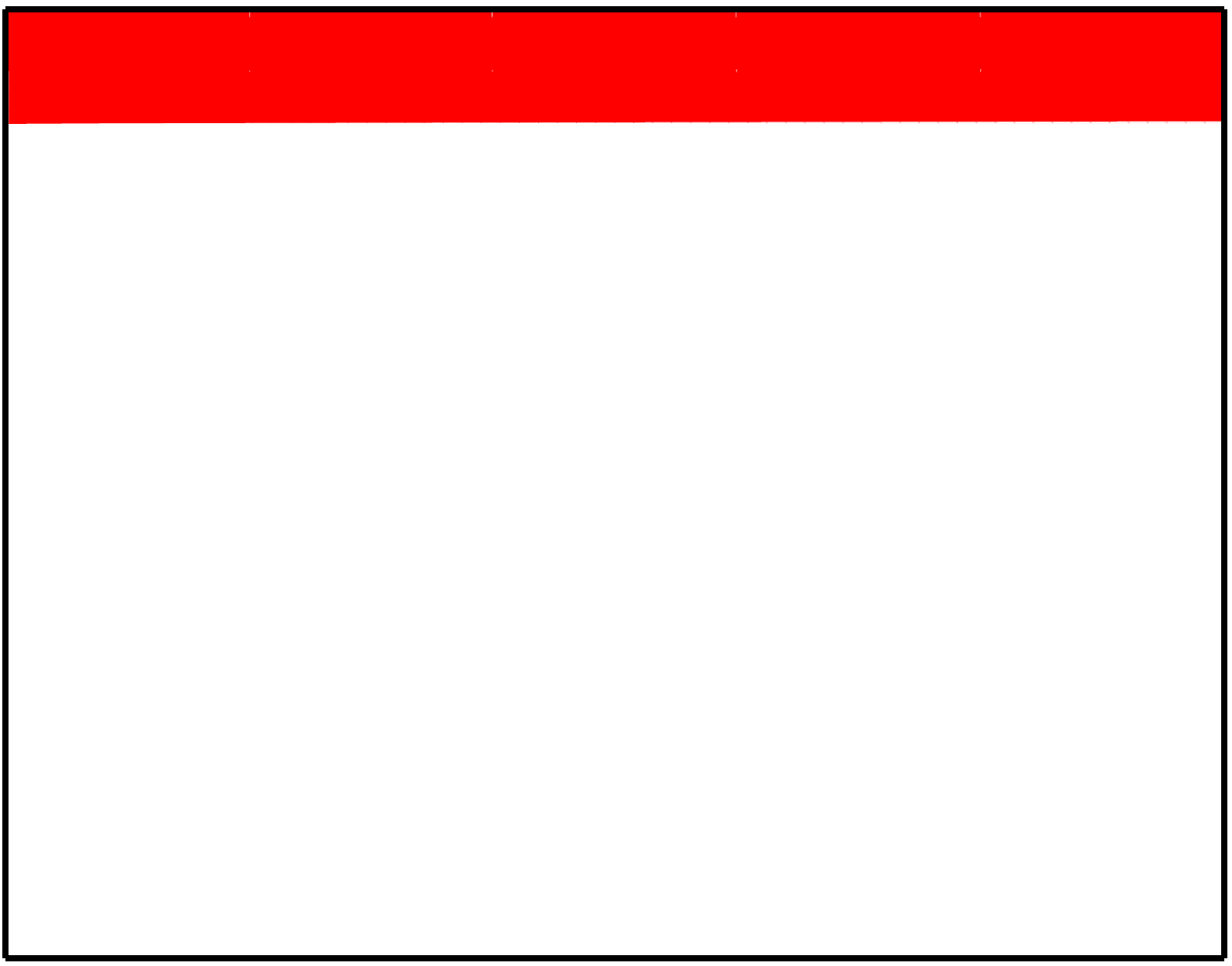}
  \caption{We consider the saw tooth domain (left). The marked area (right) is
    where the finite element space is enriched.}
  \label{fig:sawDomainArea}
\end{figure}
\begin{figure}[htb]
  \centering
  \includegraphics{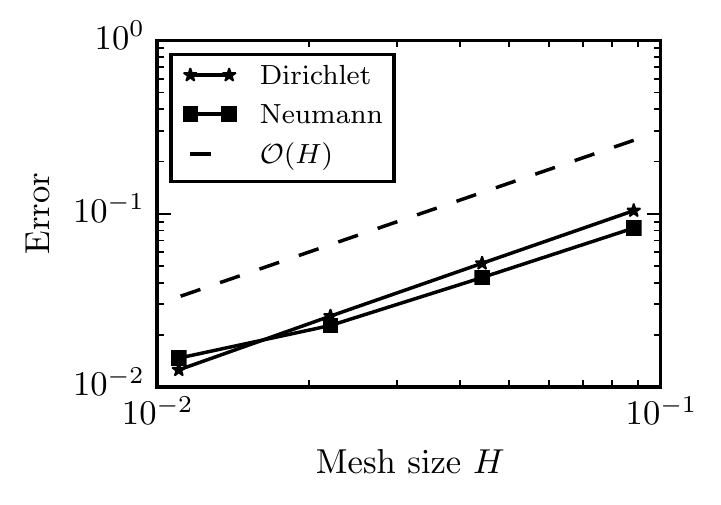}
  \includegraphics{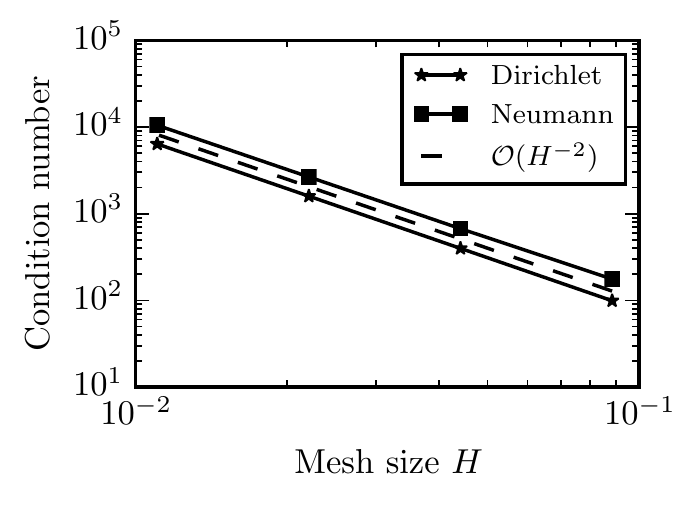}
  \caption{The convergence rate to the reference solution in relative
    energy norm (left) and the scaling of the condition number (right)
    for the saw tooth shaped boundary using different boundary
    condition on the saw tooths.}
  \label{fig:sawconvergence}
\end{figure}

\subsection{Accuracy and conditioning on cut domains}
Let us consider a domain with a re-entrant corner
$\Omega=([0,1]\times[0,1])\backslash ([0.5,1]\times[0,0.5])$. We want to
investigate how the location of the boundary relative to the coarse background mesh affects the accuracy in the approximation and the
conditioning of the matrix. We fix the coarse $H=2^{-3}$
and fine $h=2^{-8}$ mesh sizes and consider three different setups of
boundary condition, homogeneous Dirichlet on the whole boundary,
homogeneous Dirichlet on the cut elements and homogeneous Neumann
otherwise, and homogeneous Neumann on the cut elements and homogeneous
Dirichlet otherwise. We will cut the coarse mesh in two different
ways, 1) with a horizontal cut and 2) a circular cut around the
reentered corner, see Figure~\ref{fig:cut}. The errors are measured in
energy norm. The results are presented in Table~\ref{tab:cut1} and
\ref{tab:cut2}. We conclude that neither the error nor the
conditioning are sensitive to how the domain is cut.
\begin{figure}[htb]
  \centering
  \includegraphics[scale = 0.4]{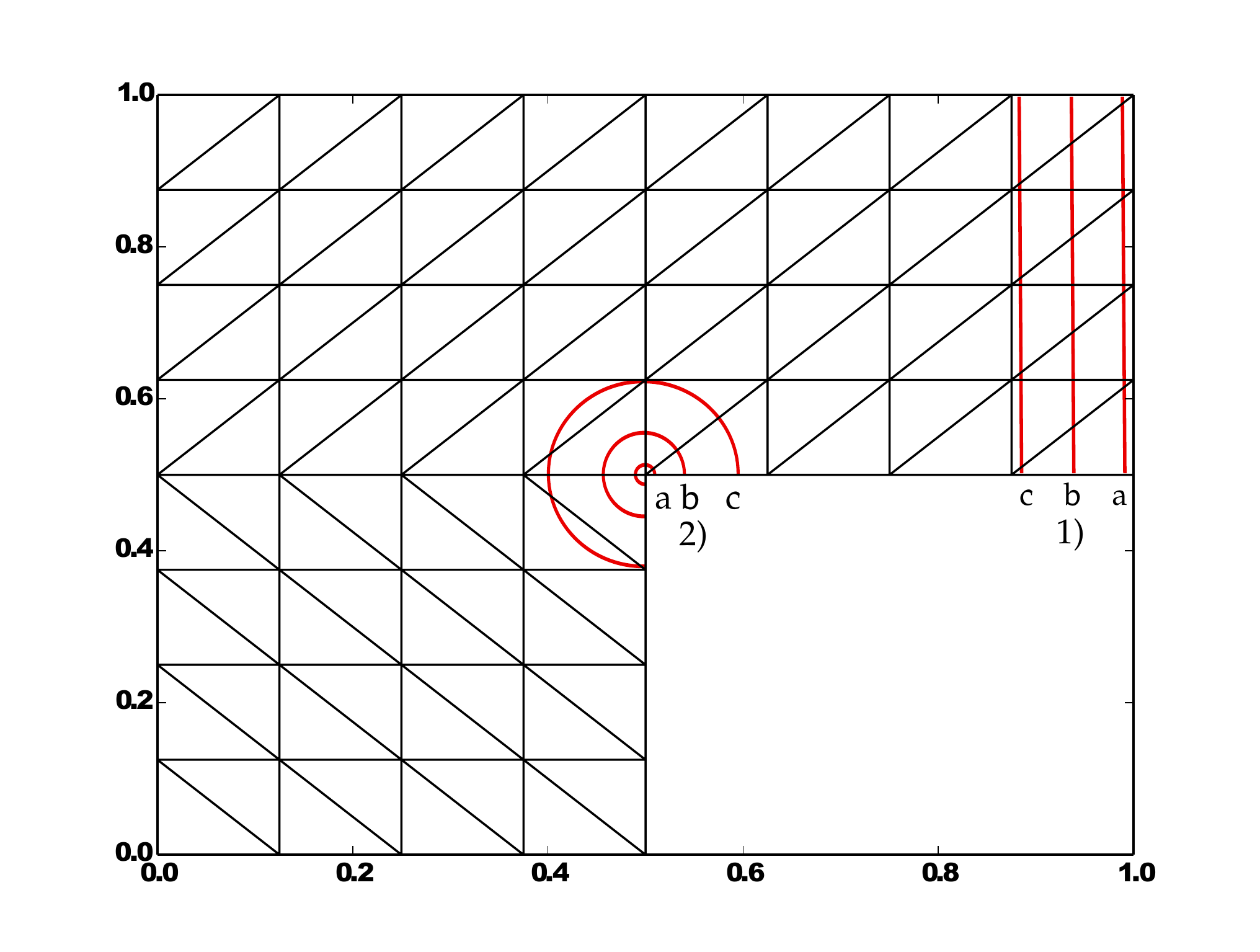}
  \caption{A given background mesh that is cut in two different ways 1) and 2) with different size of the cut a), b), and c).}
  \label{fig:cut}
\end{figure}
\begin{table}[h!tb]
  \centering
  \begin{tabular}{c|c|c|c|c}
  $\Gamma_{cut}$ & $\partial\Omega\setminus\Gamma_{cut}$    & $e_{\mathrm{rel}}(a)$ &$e_{\mathrm{rel}}(b)$ & $e_{\mathrm{rel}}(c)$ \\ \hline
    D & D & $0.059$ & $0.057$ & $0.056$ \\
    D & N & $0.018$ & $0.019$ & $0.020$ \\
    N & D & $0.063$ & $0.055$ & $0.053$  \\ \hline
   & & $\mathrm{cond}(a)$ &$\mathrm{cond}(b)$ & $\mathrm{cond}(c)$ \\ \hline
    D & D & $9.85$ & $10.10$  & $13.63$  \\
    D & N & $299.75$ & $282.26$ & $353.27$ \\
    N & D & $10.537$ & $10.79$  & $11.47$
  \end{tabular}
  \caption{For cut $1$ we have $L\cap [0,1-r]\times[0,1]$, where  $r = \{h,0.5H,H-h\}$ in $a)$, $b)$, and $c)$, respectively. The errors measured in relative energy norm and condition number of the coarse stiffness matrix are presented. We try the different boundary conditions, $D$ (Dirichlet) and $N$ (Neumann), on the boundary segement $\Gamma_{cut}$, which cuts the elements.}
  \label{tab:cut1}
\end{table}

\begin{table}[h!tb]
  \centering
  \begin{tabular}{c|c|c|c|c}
    $\Gamma_{cut}$ & $\partial\Omega\setminus\Gamma_{cut}$  & $e_{\mathrm{rel}}(a)$ &$e_{\mathrm{rel}}(b)$ & $e_{\mathrm{rel}}(c)$ \\ \hline
    D & D  & $0.060$ & $0.064$ & $0.073$ \\
    D & N & $0.0205$ & $0.035$ & $0.048$  \\
    N & D & $0.060$ & $0.057$ & $0.059$  \\ \hline
    & & $\mathrm{cond}(a)$ &$\mathrm{cond}(b)$ & $\mathrm{cond}(c)$ \\ \hline
    D & D & $9.80$   & $ 9.23$  & $7.03$ \\
    D & N & $246.99$ & $107.52$ & $59.67$ \\
    N & D & $9.90$   & $11.44$  & $12.16$
  \end{tabular}
  \caption{For cut $2$ we have $L\setminus B(x_0,r)$ for a ball $B$ centered in $x_0=(0.5,0.5)$ and with radius $r$, where  $r = \{h,0.5H,H\}$ in $a)$, $b)$, and $c)$, respectively.  The errors measured  in energy norm and condition number of the coarse stiffness matrix are presented. We try the different boundary conditions, $D$ (Dirichlet) and $N$ (Neumann), on the boundary segement $\Gamma_{cut}$, which cuts the elements.}
  \label{tab:cut2}
\end{table}

\newpage
\appendix

\section{Proofs}\label{app:proofs}
In the appendix we present the proof of Lemma \ref{lem:decay}.
\begin{proof}[Proof of Lemma~\ref{lem:decay}]
  We will make frequent use of the cut off function $\eta_x^{k-1,k}$ which
  satisfy $\eta_x^{k-1,k}=0$ in $\omega_x^{k-1}$, $\eta_x^{k-1,k}=1$ in
  $\Omega\setminus\omega_x^{k}$, and
  $\|\nabla\eta_x^{k-1,k}\|_{L^\infty(\Omega)}\lesssim H^{-1}$. Given
  $v_H\in\V_H$ with nodal values $v_x$ we define define $e=(Q - Q^L)(v_H)$ and
  $e=\sum_{x\in\N_I}e_x$ where $e_x = (Q-Q^L_x)v_x\varphi_x$, we obtain
  \begin{equation}\label{eq:decay1}
    \begin{aligned}
      |||e|||^2 &\lesssim a(e,e)=\sum_{x\in\mathcal{N}_I}a\left((Q_x -
        Q^L_x)v_x\varphi_x, e\right) = \sum_{x\in\mathcal{N}_I}a\left( e_x,e - \tilde v_x\right)
    \end{aligned}
  \end{equation}
  for $\tilde v_x$ which satisfy
  \begin{equation}\label{eq:nowtrue}
    a(e_x,\tilde v_x) = 0.
  \end{equation}
  We choose $\tilde v_x=\eta_x^{L+2,L+1}e -
  \mathcal{I}_H\eta_x^{L+2,L+1}e+g_x$ where $\mathcal{I}_H g_x = 0$,
  $g_x|_{\partial\Omega}=\mathcal{I}_H \eta_x^{L+2,L+1}e$, and $|||g_x|||\leq
  |||\mathcal{I}_H \eta_x^{L+2,L+1}e|||$. It is possible to construct $g_x$
  with the given properties, see the proof of Lemma~\ref{lem:g}. Equation
  \eqref{eq:nowtrue} now follows since $\tilde v_x\in\V^\mathrm{f}$ and
  $\supp(\tilde v_x)\cap\supp(Q^L_x(v))=0$. For each $x\in \mathcal{N}_I$, we
  obtain
  \begin{equation}
    a\left(e_x,e-\tilde v_x\right)\leq|||e_x|||\cdot|||e - \tilde v_x|||
  \end{equation}
  where we split
  \begin{equation}\label{eq:decay2}
    |||e - \tilde v_x||| \leq |||e - \eta_x^{L+2,L+1}e||| + |||\eta_x^{L+2,L+1}e - \tilde v_x|||.
  \end{equation}
  The first term in \eqref{eq:decay2} can be bounded as
  \begin{equation}\label{eq:decay3}
    \begin{aligned}
      &|||(1-\eta_T^{L+2,L+1})e|||^2_{\omega_x^{L+2}} \lesssim \|\nabla(1-\eta_T^{L+2,L+1})e\|^2_{\omega_x^{L+2}}+ \|\kappa(1-\eta_T^{L+2,L+1})e\|^2_{\partial \omega_x^{L+2}\cap \Gamma } \\
      &\qquad\leq\|\nabla e\|^2_{\omega_x^{L+2}}+ H^{-1}\|e\|^2_{\omega_x^{L+2}}+ \|\kappa(1-\eta_T^{L+2,L+1})e\|^2_{\partial \omega_x^{L+2}\cap \Gamma } \\
      &\qquad \leq |||e|||^2_{\omega_x^{L+3}}
    \end{aligned}
  \end{equation}
  using the product rule, inverse estimates, and interpolation
  estimates. The second term in \eqref{eq:decay2} can be bounded as
  \begin{equation}\label{eq:decay4}
    \begin{aligned}
      |||\eta_x^{L+2,L+1}e - \tilde v_x|||^2 = \|\nabla\mathcal{I}_H(\eta^L_Te)\|^2+ \|\kappa\mathcal{I}_H(\eta^L_Te)\|^2_{\Gamma} \lesssim |||e|||^2_{\omega^{L+4}_x}
    \end{aligned}
  \end{equation}
  where we used
  \begin{equation}
    \begin{aligned}
    \|\nabla\mathcal{I}_H(\eta_x^{L+2,L+1}e)\|^2& = \|\nabla(\mathcal{I}_H\eta_x^{L+2,L+1}e)\|^2_{\omega^{L+3}_x\setminus\omega^{L}_x}\lesssim H^{-2}\|e\|^2_{\omega^{L+3}_x\setminus\omega^{L}_x} \\
    &\leq |||e|||^2_{\omega^{L+4}_x\setminus\omega^{L-1}_x}\leq |||e|||^2_{\omega^{L+4}_x}
    \end{aligned}
  \end{equation}
   and
  \begin{equation}
    \begin{aligned}
      \|\kappa\mathcal{I}_H(\eta_x^{L+2,L+1}e)\|^2_{\Gamma_x^L} &= \sum_{E\in \Gamma_x^L} \|\kappa\mathcal{I}_H(\eta_x^{L-2,L-1}e)\|^2_{E}  \leq \kappa H^{-1} \|\mathcal{I}_H(\eta_{\omega^{L-1}_x}e)\|^2_{\omega_x^{L+3}} \\
      &\leq \kappa H^{-1} \|e-\mathcal{I}_He\|^2_{\omega_x^{L+3}}
      \leq \kappa H \|\nabla e\|^2_{\omega_x^{L+4}} \leq
      |||e|||^2_{\omega^{L+4}_x}
     \end{aligned}
  \end{equation}
  which follows from Lemma~\ref{lem:approxIH} and that $\kappa \lesssim H^{-1}$. Hence, combining
  \eqref{eq:decay1}, \eqref{eq:decay2}, \eqref{eq:decay3}, and \eqref{eq:decay4} we obtain
  \begin{equation}\label{eq:decaystepomplete}
    \begin{aligned}
      |||e|||^2 &\lesssim \sum_{x\in\mathcal{N}_I} |||e_x||| \cdot |||e |||_{\omega^{L+4}_x} \lesssim L^{d} \left(\sum_{x\in\mathcal{N}_I} |||e_x|||^2\right)^{1/2}|||e|||
    \end{aligned}
  \end{equation}
  that is
  \begin{equation}
    |||e|||^2 \lesssim L^d \sum_{x\in\mathcal{N}_I} |||e_x|||^2
  \end{equation}
  We use $\tilde w_x= (1-\eta_x^{L-1,L-2})Q(v_x\varphi_x) +
  \mathcal{I}_H\eta_x^{L-1,L-2}Q(v_x\varphi_x)+\tilde g_x\in\V^f(\omega^L_x)$ where we use the same construction of $\tilde g_x$ as for $g_x$ and obtain
  \begin{equation}
    \begin{aligned}
      |||e_x|||^2 \leq a(e_x,e_x) & = a(e_x, Q(v_x\varphi_x) -\tilde w_x) \\
      & = a(e_x,\eta_x^{L-1,L-2} Q(v_x\varphi_x) - \mathcal{I}_H\eta_x^{L-1,L-2}Q(v_x\varphi_x)+g_x) \\
      & \lesssim |||e_x|||(||| \eta_x^{L-1,L-2} Q(v_x\varphi_x)||| + |||\mathcal{I}_H\eta_x^{L-1,L-2}Q(v_x\varphi_x) |||) \\
      & \lesssim |||e_x|||\cdot|||Q(v_x\varphi_x)|||_{\Omega\setminus\omega_x^{L-3}}
    \end{aligned}
  \end{equation}
  and hence
  \begin{equation}\label{eq:decaystep1}
  	|||e_x||| \leq |v_x||||Q(\varphi_x)|||_{\Omega\setminus\omega_x^{L-3}}
  \end{equation}
  Next we construct a recursive scheme which will be used to show the
  decay. We obtain
  \begin{equation}\label{eq:decay_21}
    \begin{aligned}
      |||Q\varphi_x|||_{\Omega\setminus\omega^{k}_x}^2 &\leq \int_{\Omega} \eta_T^{k,k-1}\nabla Q\varphi_x\nabla Q\varphi_x\dx + \int_{\Gamma_R}\eta_T^{k,k-1}\kappa Q\phi_xQ\varphi_x\,\mathrm{d}S \\
      & = \int_{\Omega} \nabla Q\varphi_x\nabla (\eta_T^{k,k-1}Q\varphi_x)\dx + \int_{\Gamma_R}\kappa Q\varphi_x(\eta_T^{k,k-1}Q\varphi_x)\,\mathrm{d}S \\
      & \qquad - \int_{\Omega} Q\varphi_x\nabla Q\varphi_x \nabla \eta_T^{k,k-1}\dx \\
      & = a(Q\varphi_x, \eta_T^{k,k-1}Q\varphi_x) - \int_\Omega Q\varphi_x\nabla Q\varphi_x \nabla \eta_T^{k,k-1}\dx.
    \end{aligned}
  \end{equation}
  The first term in \eqref{eq:decay_21} can be bounded as
  \begin{equation}\label{eq:decay_22}
    \begin{aligned}
      a(Q\varphi_x, \eta_x^{k,k-1}Q\varphi_x) &= a(Q\varphi_x, \eta_x^{k,k-1}Q\varphi_x-\mathcal{I}_H\eta_x^{k,k-1}Q\varphi_x-\hat g_x)+a(Q\varphi_x,\mathcal{I}_H\eta_x^{k,k-1}Q\varphi_x+\hat g_x) \\
      & = a(Q\varphi_x,\mathcal{I}_H\eta_x^{k,k-1}Q\varphi_x+\hat g_x) \lesssim |||Q\varphi_x|||_{\omega^k_x\setminus\omega^{k-1}_x}|||\mathcal{I}_H\eta_x^{k,k-1}Q\varphi_x ||| \\
      & \lesssim |||Q\varphi_x|||_{\omega^k_x\setminus\omega^{k-1}_x}|||Q\varphi_x ||||_{\omega^{k+1}_x\setminus\omega^{k-2}_x}|||Q\varphi_x |||_{\omega^{k+1}_x\setminus\omega^{k-2}_x}^2
    \end{aligned}
  \end{equation}
  where again we construct $\hat g_x$ such that $|||\hat g_x|||\lesssim
 |||\mathcal{I}_H\eta_x^{k,k-1}Q\varphi_x |||$ and  use
 Lemma~\ref{lem:approxIH}. The second term is bounded as
  \begin{equation}\label{eq:decay_23}
    \begin{aligned}
      \int_{\Omega} Q\varphi_x\nabla Q\varphi_x \nabla \eta_x^{k,k-1}\dx &\leq H^{-1}\|Q\varphi_x- \mathcal{I}_HQ_x v\|_{\omega^k_x\setminus\omega^{k-1}_x}|||Q\varphi_x|||_{\omega^k_x\setminus\omega^{k-1}_x} \\
      &\lesssim |||Q\varphi_x |||_{\omega^{k+1}_x\setminus\omega^{k-2}_x}|||Q\varphi_x ||||_{\omega^{k+1}_x\setminus\omega^{k-2}_x} \leq |||Q\varphi_x |||_{\omega^{k+1}_x\setminus\omega^{k-2}_x}^2
    \end{aligned}
  \end{equation}
  Combining \eqref{eq:decay_21}, \eqref{eq:decay_22}, and
  \eqref{eq:decay_23} we obtain
  \begin{equation}
    \begin{aligned}
    |||Q\varphi_x|||_{\Omega\setminus\omega^k_x}^2 &\leq C_1 |||Q\varphi_x |||_{\omega^{k+1}_x\setminus\omega^{k-2}_x}^2 = C_1\left( |||Q\varphi_x |||_{\Omega\setminus\omega^{k-2}_x}^2-|||Q\varphi_x |||_{\Omega_x\setminus\omega^{k+1}_x}^2\right) \\
    &\leq C_1\left( |||Q\varphi_x |||_{\Omega\setminus\omega^{k-2}_x}^2-|||Q\varphi_x |||_{\Omega_x\setminus\omega^{k}_x}^2\right)
    \end{aligned}
  \end{equation}
  and hence
  \begin{equation}\label{eq:recursively}
    |||Q(v_x\varphi_x)|||_{\Omega\setminus\omega^k_x}^2 \leq \gamma|||Q(v_x\varphi_x)|||_{\Omega\setminus\omega^{k-2}_x}^2,
  \end{equation}
  where $\gamma = \frac{C_1}{1+C_1}$. Using \eqref{eq:recursively} recursively we obtain
  \begin{equation}\label{eq:recursively2}
    \begin{aligned}
&      |||Q\varphi_x|||_{\Omega\setminus\omega^{L-3}}^2 \leq \gamma^{k}|||Q\varphi_x|||_{\Omega\setminus\omega^{L-3(k+1)}_x}^2 \\
\Leftrightarrow \quad& |||Q\varphi_x|||_{\Omega\setminus\omega^{L-3}}^2 \leq \gamma^{\left\lfloor (L-3)/3 \right\rfloor}|||Q\varphi_x|||_{\Omega\setminus\omega^{0}_x}^2 \leq \gamma^{\left\lfloor (L-3)/3 \right\rfloor}|||Q\varphi_x|||^2.
    \end{aligned}
  \end{equation}
  Combing \eqref{eq:decaystepomplete}, \eqref{eq:decaystep1}, and, \eqref{eq:recursively2} we get
  \begin{equation}
    ||| e |||^2 \lesssim L^{d}\gamma^{\left\lfloor (L-3)/3 \right\rfloor}\sum_{x\in\mathcal{N}_I} v_x^2|||Q\varphi_x|||^2
  \end{equation}
  which concludes the lemma.
\end{proof}

\bibliographystyle{abbrv}
\bibliography{references}
\end{document}